\documentclass[12pt,reqno]{amsart}

\usepackage{amssymb,hyperref}

\usepackage[all]{xy}

\newtheorem{theorem}{Theorem}[section]
\newtheorem{proposition}[theorem]{Proposition}
\newtheorem{lemma}[theorem]{Lemma}
\newtheorem{corollary}[theorem]{Corollary}

\theoremstyle{definition}
\newtheorem{definition}[theorem]{Definition}
\newtheorem{remark}[theorem]{Remark}

\numberwithin{equation}{section}

\begin{document}

\baselineskip=15pt

\title[Parabolic opers and differential operators]{Parabolic opers and differential operators}

\author[I. Biswas]{Indranil Biswas}

\address{School of Mathematics, Tata Institute of Fundamental
Research, Homi Bhabha Road, Mumbai 400005, India}

\email{indranil@math.tifr.res.in}

\author[N. Borne]{Niels Borne}

\address{Universit\'e Lille 1, Cit\'e scientifique
U.M.R. CNRS 8524, U.F.R. de Math\'ematiques
59 655 Villeneuve d'Ascq C\'edex, France}

\email{Niels.Borne@math.univ-lille1.fr}

\author[S. Dumitrescu]{Sorin Dumitrescu}

\address{Universit\'e C\^ote d'Azur, CNRS, LJAD, France}

\email{dumitres@unice.fr}

\author[S. Heller]{Sebastian Heller}

\address{Beijing Institute of Mathematical Sciences and Applications,
Yanqi Island, Huairou District, Beijing 101408}

\email{sheller@bimsa.cn}

\author[C. Pauly]{Christian Pauly}

\address{Universit\'e C\^ote d'Azur, CNRS, LJAD, France}

\email{pauly@unice.fr}



\date{}

\begin{abstract}
Parabolic ${\rm SL}(r,\mathbb C)$--opers were defined and investigated in \cite{BDP} in the 
set-up of vector bundles on curves with a parabolic structure over a divisor. Here we 
introduce and study holomorphic differential operators between parabolic vector bundles over 
curves. We consider the parabolic ${\rm SL}(r,\mathbb C)$--opers on a Riemann 
surface $X$ with given singular divisor $S \,\subset\, X$ and with fixed parabolic weights satisfying
the condition that all parabolic weights at any $x_i\, \in\, S$ are
integral multiples of $\frac{1}{2N_i+1}$, where $N_i\,>\,1$ are fixed integers. We prove that this
space of opers is canonically identified with the affine space of holomorphic differential 
operators of order $r$ between two natural parabolic line bundles on $X$ (depending only on the divisor $S$ 
and the weights $N_i$) satisfying the conditions that the principal symbol of the
differential operators is the constant function $1$ and the sub-principal symbol vanishes identically.
The vanishing of the sub-principal symbol ensures that the logarithmic connection on the rank $r$ bundle
is actually a logarithmic ${\rm SL}(r,\mathbb C)$--connection. 
\end{abstract}

\maketitle

\tableofcontents

\section{Introduction}\label{se1}

After the seminal work of Drinfeld and Sokolov \cite{DS1}, \cite{DS2}, the notion of opers was 
introduced by Beilinson and Drinfeld \cite{BD1, BD2} as geometric structures on Riemann 
surfaces that formalize the notion of ordinary differential equations in a coordinate-free 
way. This broad formalism encapsulates the classical notion of a Riccati equation, or 
equivalently that of a complex projective structure on a Riemann surface, as being an ${\rm 
SL}(2,\mathbb C)$--oper. Since then the notion of oper turned out to be very important, not 
only in the study of differential equations, but also in very diverse topics, as for example, 
geometric Langlands correspondence, nonabelian Hodge theory and also some branches of 
mathematical physics; see, for example, \cite{BF}, \cite{DFKMMN}, \cite{FT}, \cite{FG}, 
\cite{FG2}, \cite{CS}, \cite{Fr}, \cite{Fr2}, \cite{BSY} and references therein. In contemporary
research in mathematics and mathematical physics, 
the study of opers and their applications have been firmly established as an important topic,
testified by the works of many. In particular, important progress in the understanding of opers was 
carried out in \cite{BD1, BD2, FG, FG2, AB, W, ABF, In, IIS1, IIS2}.

In \cite{BDP}, three of the authors introduced and studied parabolic ${\rm SL}(r, {\mathbb 
C})$--opers on curves in the set-up of parabolic vector bundles as defined by Mehta and 
Seshadri, \cite{MS}, and also by Maruyama and Yokogawa \cite{MY}.

Later on, being inspired by the works \cite{AB, Sa}, the 
infinitesimal deformations of parabolic ${\rm SL}(r, {\mathbb C})$--opers and also the monodromy map
for parabolic ${\rm SL}(r, {\mathbb C})$--opers
were studied in \cite{BDHP}. It may be mentioned that the appendix of \cite{BDHP} provides an alternative 
definition of a parabolic ${\rm SL}(r,{\mathbb C})$--oper in terms of $\mathbb{R}$-filtered 
sheaves as introduced and studied by Maruyama and Yokogawa in \cite{MY}. This definition is 
conceptually closer to the definition of an ordinary ${\rm SL}(r,{\mathbb C})$--oper and 
clarifies the one given in \cite{BDP}.

The objective of this article is to further investigate parabolic ${\rm SL}(r, {\mathbb C})$--opers 
and to characterize them as a special class of holomorphic differential operators on parabolic 
bundles. It should be recalled that the relation between opers and differential operators is 
established and well-known in the context of ordinary opers \cite{BD1}. Here we introduce and 
study holomorphic differential operators on parabolic vector bundles over Riemann surfaces 
under the condition that at each point $x_i$ on the singular divisor $S$ all the parabolic 
weights are integral multiples of $\frac{1}{2N_i+1}$, with $N_i \,>\,1$ being an integer. Under this 
assumption, the main result of the article, Theorem \ref{thm1}, proves that the space of all 
parabolic ${\rm SL}(r,{\mathbb C})$--opers on $X$ with given singular set $S\,:=\, \{x_1,\, 
\cdots,\, x_n\}\, \subset\, X$ and fixed parabolic weights integral multiples of 
$\frac{1}{2N_i+1}$ at each $x_i \,\in\, S$, is canonically identified with the affine space of $r$-order 
holomorphic differential operators between two natural parabolic line bundles on $X$ 
(depending only on $S$ and the weights $N_i$) having as principal symbol the constant function 
$1$ and with vanishing sub-principal symbol. The vanishing of the sub-principal symbol ensures that
the logarithmic connection on the rank $r$ bundle is indeed a logarithmic ${\rm SL}(r, {\mathbb C})$--connection.

The article is organized in the following way. Section \ref{se2} deals with parabolic ${\rm 
SL}(2, {\mathbb C})$--opers. In particular we introduce a rank two parabolic bundle which is a 
parabolic version of the indigenous bundle (also called Gunning bundle or uniformization 
bundle) introduced in \cite{Gu} (see also \cite{De}); recall that this
indigenous bundle introduced by Gunning is the rank two holomorphic vector bundle 
associated to any ordinary ${\rm SL}(2, {\mathbb C})$--oper (e.g. a complex projective 
structure) on a given Riemann surface. It should be clarified that this parabolic analog of Gunning 
bundle depends only on the divisor $S$ and the integers $N_i$. All parabolic ${\rm SL}(2, 
{\mathbb C})$--opers with given singular set $S$ and fixed weights are parabolic connections 
on the same parabolic Gunning bundle.

Section \ref{se3} starts with an explicit description of several (parabolic) symmetric powers 
of the rank two parabolic Gunning bundle constructed in Section \ref{se2}; then ${\rm SL}(r, 
{\mathbb C})$--opers on a Riemann surface $X$, singular over $S \,\subset\, X$, are defined (see 
Definition (\ref{def1})). In this context Proposition \ref{prop4} proves that parabolic ${\rm 
SL}(r, {\mathbb C})$--opers on $X$ with weights equal to integral multiples of $\frac{1}{2N_i+1}$ 
at each $x_i \,\in\, S$ are in natural bijection with invariant ${\rm SL}(r, {\mathbb C})$--opers 
on a ramified Galois covering $Y$ over $X$ equipped with an action of the Galois group. This 
Proposition \ref{prop4} is a generalization of Theorem 6.3 in \cite{BDP} where a similar 
result was proved under the extra assumption that $r$ is odd. The proof of Proposition 
\ref{prop4} uses in an essential way the correspondence studied in \cite{Bi}, \cite{Bo1}, 
\cite{Bo2}, and also a result (Corollary \ref{cor1}(3)) of Section \ref{se2} proving that, at 
each point of $S$, the monodromy of any parabolic connection on the parabolic Gunning bundle 
is semisimple.

Section \ref{se4} constructs the canonical parabolic filtration associated to any parabolic 
${\rm SL}(r, {\mathbb C})$--oper. This parabolic filtration depends only on $S$ and the 
integers $N_i$. It is proved then that any parabolic connection on the associated parabolic 
bundle satisfies the Griffith transversality condition with respect to the above filtration 
(all corresponding second fundamental forms are actually isomorphisms).
 
Section \ref{se5} defines and study several equivalent definitions for holomorphic 
differential operators between parabolic vector bundles. Under the above rationality 
assumption on the parabolic weights, Proposition \ref{prop5} proves that holomorphic 
differential operators between parabolic vector bundles are canonically identified with the
invariant holomorphic differential operators between corresponding orbifold vector bundles on a ramified Galois 
covering $Y$ over $X$ equipped with an action of the Galois group. We deduce the construction 
of the principal symbol map defined on the space of differential operators in the parabolic 
set-up (see Lemma \ref{lem7}).

The last Section focuses on the class of holomorphic differential operators associated to ${\rm SL}(r, {\mathbb C})$--opers. These 
are holomorphic differential operators between two parabolic line bundles over $X$ naturally associated to the Gunning parabolic 
bundle (those line bundles only depend on the divisor $S$ and the parabolic weights $N_i$). In this case the principal symbol is the 
constant function $1$ and the sub-principal symbol map (constructed in Lemma \ref{lem8}) defined on the space of parabolic 
differential operators between the appropriate parabolic line bundles vanishes. Then the main Theorem \ref{thm1} stated above is 
proved.

\section{A rank two parabolic bundle}\label{se2}

Let $X$ be a compact connected Riemann surface. Its canonical
line bundle will be denoted by $K_X$. Fix a finite subset of $n$
distinct points
\begin{equation}\label{e1}
S\, :=\, \{x_1,\, \cdots,\, x_n\}\, \subset\, X.
\end{equation}
The reduced effective divisor $x_1+\ldots + x_n$ on $X$ will also be denoted
by $S$.

If ${\rm genus}(X)\,=\, 0$, we assume that $n\, \geq \,3$.

For any holomorphic vector bundle $E$ on $X$, and any $k\, \in\, \mathbb Z$, the holomorphic 
vector bundle $E\otimes{\mathcal O}_X(kS)$ on $X$ will be denoted by $E(kS)$.

Let us first start with the definition of a parabolic structure on a holomorphic vector bundle 
over $X$ having $S$ as the parabolic divisor.

\subsection{Parabolic bundles and parabolic connections}

A quasiparabolic structure on a holomorphic vector bundle $E$ on $X$, associated to the 
divisor $S$, is a filtration of subspaces of the fiber $E_{x_i}$ of $E$ over $x_i$
\begin{equation}\label{ec2}
E_{x_i}\,=\,E_{i,1}\,\supset\, E_{i,2}\,\supset\, \cdots\,
\supset \,E_{i,l_i} \,\supset\, E_{i,l_i+1}\,=\, 0
\end{equation}
for every $1\, \leq\, i\, \leq\, n$. A parabolic structure on
$E$ is a quasiparabolic structure as above together with a finite sequence of positive real numbers
\begin{equation}\label{e2a}
0\,\leq\, \alpha_{i,1} \,< \,\alpha_{i,2} \,<\,
\cdots \,<\, \alpha_{i,l_i}\,<\, 1
\end{equation}
for every $1\, \leq\, i\, \leq\, n$. The number $\alpha_{i,j}$ is called the parabolic
weight of the corresponding subspace $E_{i,j}$ in \eqref{ec2} (see \cite{MS}, \cite{MY}).

A parabolic vector bundle is a holomorphic vector bundle $E$ with a parabolic structure
$(\{E_{i,j}\},\, \{\alpha_{i,j}\})$. It will be denoted by $E_*$ for convenience.

A \textit{logarithmic connection} on the holomorphic vector bundle $E$, singular over
$S$, is a holomorphic differential operator of order one
$$
D\, :\, E\, \longrightarrow\, E \otimes K_X\otimes {\mathcal O}_X(S)
$$
satisfying the Leibniz rule, meaning
\begin{equation}\label{lr}
D(fs)\,=\, fD(s)+ s\otimes df
\end{equation}
for any locally
defined holomorphic function $f$ on $X$ and any locally defined holomorphic section $s$ of $E$.

Recall that any logarithmic connection on $E$ over the Riemann surface is necessarily flat. Indeed, the curvature ($2$-form) vanishes 
identically because $\Omega^{2,0}_X\,=\,0$.

Take a point $x_i\, \in\, S$. The fiber of
$K_X\otimes{\mathcal O}_X(S)$ over $x_i$ is identified with $\mathbb C$ by the Poincar\'e
adjunction formula \cite[p.~146]{GH} which gives an isomorphism
\begin{equation}\label{af}
{\mathcal O}_X(-x_i)_{x_i }\,\stackrel{\sim}{\longrightarrow}\,(K_X)_{x_i}.
\end{equation}
To describe this isomorphism, let $z$ be a holomorphic coordinate function on $X$ defined on
an analytic open neighborhood of $x_i$ such that $z(x_i)\,=\, 0$. We have an isomorphism ${\mathcal O}_X
(-x_i)_{x_i} \,\longrightarrow\, (K_X)_{x_i}$ that sends $z$ to $dz(x_i)$. It is straightforward to check that 
this map is actually independent of the choice of the holomorphic local coordinate $z$ at $x_i$.

Let $D\, :\, E\, \longrightarrow\, E \otimes K_X\otimes{\mathcal O}_X(S)$ be a logarithmic
connection on $E$. From \eqref{lr} it follows that the composition of homomorphisms
\begin{equation}\label{ec1}
E \, \xrightarrow{\,\ D\,\ }\, E \otimes K_X\otimes{\mathcal O}_X(S) \, \longrightarrow\,
(E\otimes K_X\otimes{\mathcal O}_X(S))_{x_i}\,\stackrel{\sim}{\longrightarrow}\, E_{x_i}
\end{equation}
is ${\mathcal O}_X$--linear; the above isomorphism
$(E\otimes K_X\otimes{\mathcal O}_X(S))_{x_i}\,
\stackrel{\sim}{\longrightarrow}\, E_{x_i}$ is given by the isomorphism in \eqref{af}.
Therefore, the composition of homomorphisms in \eqref{ec1} produces a
$\mathbb C$--linear homomorphism
\begin{equation}\label{er}
{\rm Res}(D\,,x_i)\, :\, E_{x_i}\, \longrightarrow\, E_{x_i}\, ,
\end{equation}
which is called the \textit{residue} of the logarithmic connection $D$ at $x_i$ (see \cite{De} 
for more details).

\begin{remark}\label{rem1}
The local monodromy of $D$ around $x_i$ is conjugated to $$\exp\left(-2\pi\sqrt{-1}\cdot {\rm 
Res}(D,\, x_i)\right)\,\in\, {\rm GL}(E_{x_i})$$
\cite{De}.
\end{remark}

Consider now $E$ with its parabolic structure $E_*\,=\,(E,\, (\{E_{i,j}\},\, \{\alpha_{i,j}\}))$;
see \eqref{ec2}, \eqref{e2a}.

A \textit{parabolic connection} on $E_*$ is a logarithmic connection $D$ on $E$, singular over
$S$, such that
\begin{enumerate}
\item $\text{Res}(D,x_i)(E_{i,j})\, \subset\, E_{i,j}$ for all $1\,\leq\, j\,\leq\, l_i$,
$1\,\leq\, i\, \leq\, n$ (see \eqref{ec2}), and

\item the endomorphism of $E_{i,j}/E_{i,j+1}$ induced by $\text{Res}(D,x_i)$ coincides with
multiplication by the parabolic weight $\alpha_{i,j}$ for all $1\,\leq\, j\,\leq\, l_i$,
$1\,\leq\, i\, \leq\, n$ (see \eqref{e2a}).
\end{enumerate}

\begin{remark}\label{re-cc}
The following necessary and sufficient condition for
$E_*$ to admit a parabolic connection was given in \cite{BL}:

A parabolic vector bundle $E_*$ admits a parabolic connection if and only if the parabolic degree of every direct summand of
$E_*$ is zero \cite[p.~594, Theorem 1.1]{BL}.
\end{remark}

\subsection{The parabolic Gunning bundle}
 
Choose a holomorphic line bundle ${\mathcal L}$ on $X$ such that ${\mathcal L}^{\otimes 2}$
is holomorphically isomorphic to $K_X$; also fix a holomorphic isomorphism between
${\mathcal L}^{\otimes 2}$ and $K_X$.

We have $H^1(X,\, {\rm Hom}({\mathcal L}^*,\, {\mathcal L}))\,=\, H^1(X,\, K_X)\,=\, H^0(X,\, {\mathcal O}_X)^*
\,=\, {\mathbb C}$ (Serre duality); note that here the chosen isomorphism between
${\mathcal L}^{\otimes 2}$ and $K_X$ is being used. Consequently,
there is a natural nontrivial extension $\widetilde{E}$ of
${\mathcal L}^*$ by $\mathcal L$ that corresponds to $$1\, \in\, H^1(X,\, {\rm Hom}({\mathcal L}^*,\,
{\mathcal L})).$$ So $\widetilde{E}$ fits in a short exact sequence of holomorphic vector bundles
\begin{equation}\label{e3}
0\, \longrightarrow\, {\mathcal L} \, \longrightarrow\, \widetilde{E} 
\, \xrightarrow{\,\,p_0\,\,}\, {\mathcal L}^* \, \longrightarrow\, 0\, ;
\end{equation}
this short exact sequence does not split holomorphically.
Consider the subsheaf ${\mathcal L}^* (-S)\,\subset\, {\mathcal L}^*$. Define
$$
E\, :=\, p^{-1}_0({\mathcal L}^*(-S)) \, \subset\, \widetilde{E}\, ,
$$
where $p_0$ is the projection in \eqref{e3}. From \eqref{e3} we know that this $E$ fits in a short exact
sequence of holomorphic vector bundles
\begin{equation}\label{e4}
0\, \longrightarrow\, {\mathcal L} \, \stackrel{\iota}{\longrightarrow}\, E 
\, \stackrel{p}{\longrightarrow}\, {\mathcal L}^*(-S) \, \longrightarrow\, 0\, ;
\end{equation}
the projection $p$ in \eqref{e4} is the restriction, to the subsheaf $E$, of $p_0$ in \eqref{e3}.

\begin{lemma}\label{lem1}
Take any point $x\, \in\, S$. The fiber $E_x$ of $E$ (see \eqref{e4}) over $x$ canonically decomposes as
$$
E_x\,=\, {\mathcal L}_x \oplus {\mathcal L}^*(-S)_x\,=\, {\mathcal L}_x\oplus {\mathcal L}_x\, .
$$
\end{lemma}

\begin{proof}
Take $x\, \in\, S$. First we have the homomorphism
\begin{equation}\label{e5}
\iota(x)\, :\, {\mathcal L}_x \, \longrightarrow\, E_x\, ,
\end{equation}
where $\iota$ is the homomorphism in \eqref{e4}, which is evidently injective. On the other hand, tensoring
\eqref{e3} with ${\mathcal O}_X(-S)$ and using the natural map of it to \eqref{e4} we have the commutative diagram
\begin{equation}\label{e6}
\begin{matrix}
0 & \longrightarrow & {\mathcal L}(-S) & \stackrel{\iota'}{\longrightarrow} & \widetilde{E}(-S)
& \stackrel{p'}{\longrightarrow} & {\mathcal L}^*(-S) & \longrightarrow & 0\\
&&\,\,\, \Big\downarrow\psi' &&\,\,\, \Big\downarrow\psi &&\,\,\,\,\Big\downarrow{\rm Id}\\
0 & \longrightarrow & {\mathcal L} & \stackrel{\iota}{\longrightarrow} & E 
& \stackrel{p}{\longrightarrow} & {\mathcal L}^*(-S) & \longrightarrow & 0,
\end{matrix}
\end{equation}
where $\iota'$ and $p'$ are the restrictions of $\iota$ and $p$ respectively.
Note that the composition of maps $$\psi(x)\circ\iota'(x)\, :\, {\mathcal L}(-S)_x\, \longrightarrow\, E_x$$
in \eqref{e6} is the zero homomorphism, because $\psi'(x)\,:\, {\mathcal L}(-S)_x\, \longrightarrow\, {\mathcal L}_x$ is
the zero homomorphism and $\psi\circ\iota'\,=\, \iota\circ\psi'$ by the commutativity of
\eqref{e6}. Since $\psi(x)\circ\iota'(x)\,=\, 0$, the homomorphism $\psi(x)$ is given by a homomorphism
\begin{equation}\label{e7}
q_x\, :\, \widetilde{E}(-S)_x/(\iota'(x)({\mathcal L}(-S)_x))\, =\, {\mathcal L}^*(-S)_x \, \longrightarrow\, E_x\, .
\end{equation}
The homomorphism $q_x$ in \eqref{e7} is injective, because $\psi(x)\, \not=\, 0$. From \eqref{e5} and \eqref{e7} we have
\begin{equation}\label{di}
\iota(x)\oplus q_x\, :\, {\mathcal L}_x\oplus {\mathcal L}^*(-S)_x \, \longrightarrow\, E_x
\end{equation}
which is clearly an isomorphism.

Using \eqref{af} and the given isomorphism between ${\mathcal L}^{\otimes 2}$ and $K_X$ we have
$$
{\mathcal L}^*(-S)_x \,=\, ((K_X)_x\otimes {\mathcal L}^*_x)^*\otimes {\mathcal O}_X(-S)_x
\,=\, ({\mathcal L}^*_x)^* \,=\, {\mathcal L}_x\, .
$$
Hence the isomorphism in \eqref{di} gives that $E_x\,=\, {\mathcal L}_x\oplus {\mathcal L}^*(-S)_x\,=\,
{\mathcal L}_x\oplus {\mathcal L}_x$.
\end{proof}

For each $x_i\,\in\, S$ (see \eqref{e1}), fix
\begin{equation}\label{e2}
c_i\, \in\, {\mathbb R}
\end{equation}
such that $c_i\, >\, 1$. Using $\{c_i\}_{i=1}^n$ we will construct a parabolic structure on the
holomorphic vector bundle $E$ in \eqref{e4}.

For any $x_i\,\in\, S$, the quasiparabolic filtration of $E_{x_i}$ is the following:
\begin{equation}\label{e8}
0\, \subset\, {\mathcal L}^*(-S)_{x_i}\, \subset\, E_{x_i}
\end{equation}
(see Lemma \ref{lem1}). The parabolic weight of ${\mathcal L}^*(-S)_{x_i}$ is $\frac{c_i+1}{2c_i+1}$; the parabolic
weight of $E_{x_i}$ is $\frac{c_i}{2c_i+1}$. The parabolic vector bundle defined by this parabolic structure on $E$
will be denoted by $E_*$. Note that
\begin{equation}\label{e9}
\text{par-deg}(E_*)\,=\, \text{degree}(E)+\sum_{i=1}^n \left(\frac{c_i+1}{2c_i+1}+ \frac{c_i}{2c_i+1}\right)\,=\,
-n+n\,=\, 0\, ;
\end{equation}
in fact the parabolic second exterior product is
\begin{equation}\label{e10}
\det E_* \,=\, \bigwedge\nolimits^2 E_*\,=\, (\bigwedge\nolimits^2 E)\otimes{\mathcal O}_X(S)
\,=\, {\mathcal O}_X\, ,
\end{equation}
where ${\mathcal O}_X$ is equipped with the trivial parabolic structure (no nonzero parabolic weights).

\begin{proposition}\label{prop1}
\mbox{}
\begin{enumerate}
\item The holomorphic vector bundle $E$ in \eqref{e4} is isomorphic to a direct sum of holomorphic line bundles
${\mathcal L}\oplus {\mathcal L}^*(-S)$.

\item The parabolic vector bundle $E_*$ in \eqref{e8} is not isomorphic to a direct sum of parabolic line bundles.
\end{enumerate}
\end{proposition}

\begin{proof}
Consider the short exact sequence in \eqref{e4}. Note that
$$
H^1(X,\, \text{Hom}({\mathcal L}^*(-S),\, {\mathcal L}))\,=\, H^1(X,\, K_X(S))\,=\,
H^0(X,\, {\mathcal O}_X(-S))^*\,=\, 0\, .
$$
Hence the short exact sequence in \eqref{e4} splits holomorphically, and $E\,=\,
{\mathcal L}\oplus {\mathcal L}^*(-S)$. This proves the
first statement.

To prove the second statement by contradiction, assume that
\begin{equation}\label{e11}
E_*\,=\, A_*\oplus B_*\, ,
\end{equation}
where $A_*$ and $B_*$ are parabolic line bundles on $X$. Since
$$
\text{par-deg}(A_*)+\text{par-deg}(B_*)\,=\, \text{par-deg}(E_*)\,=\, 0
$$
(see \eqref{e9}), at least one of $A_*$ and $B_*$ has nonnegative parabolic degree.
Assume that $\text{par-deg}(A_*)\, \geq\, 0$. Since
the parabolic degree of the quotient ${\mathcal L}^*(-S)$ in \eqref{e4}, equipped with the parabolic
structure induced by $E_*$, is negative (recall that $n\,\geq\, 3$ if 
${\rm genus}(X)\,=\, 0$), there is no nonzero homomorphism from $A_*$ to it (recall
that $\text{par-deg}(A_*)\, \geq\, 0$). Consequently, the parabolic subbundle $A_*\, \subset\, E_*$
in \eqref{e11} coincides with the subbundle $\mathcal L$ in \eqref{e4} equipped with the parabolic
structure induced by $E_*$. This implies that the following composition of homomorphisms
$$
B\, \hookrightarrow\, E \, \longrightarrow\, E/{\mathcal L}\,=\, {\mathcal L}^*(-S)
$$
is an isomorphism, where $B$ denotes the holomorphic line bundle underlying $B_*$ in \eqref{e11}. Therefore,
the inclusion map $B\, \hookrightarrow\, E$ in \eqref{e11} produces a holomorphic splitting
\begin{equation}\label{e12}
\rho\,:\, {\mathcal L}^*(-S)\, \longrightarrow\, E
\end{equation}
of \eqref{e4}. Since $\rho$ in \eqref{e12} is given by \eqref{e11}, and 
the parabolic subbundle $A_*\, \subset\, E_*$
in \eqref{e11} coincides with the subbundle $\mathcal L$ in \eqref{e4} equipped with the parabolic
structure induced by $E_*$, it follows that for all $x\, \in\, S$,
\begin{equation}\label{e13}
\rho({\mathcal L}^*(-S)_x)\,=\, {\mathcal L}^*(-S)_x\, \subset\, E_x\, .
\end{equation}
Recall that the quasiparabolic structure of $E_*$ at $x$ is given by 
the subspace ${\mathcal L}^*(-S)_x\, \subset\, E_x$ in Lemma \ref{lem1}, and therefore ${\mathcal L}^*(-S)_x$
must lie in the image, in $E_*$, of either $A_*$ or $B_*$.

{}From \eqref{e13} it follows that $\rho$ in \eqref{e12} satisfies the condition
$$
\rho({\mathcal L}^*(-S))\, \subset\, \psi(\widetilde{E}(-S))\, \subset\, E\, ,
$$
where $\psi$ is the homomorphism in \eqref{e6}. Consequently, $\rho$ produces a unique holomorphic
homomorphism $$\alpha\,:\, {\mathcal L}^*(-S)\, \longrightarrow\, \widetilde{E}(-S)$$
such that $\rho\,=\, \psi\circ\alpha$ on ${\mathcal L}^*(-S)$.
This homomorphism $\alpha$ evidently gives a holomorphic splitting of the
top exact sequence in \eqref{e6}, meaning $p'\circ\iota\,=\, {\rm Id}_{{\mathcal L}^*(-S)}$,
where $p'$ is the projection in \eqref{e6}. After tensoring the above
homomorphism $\alpha$ with ${\rm Id}_{{\mathcal O}_X(S)}$ we get a homomorphism
$$
{\mathcal L}^*\,=\, {\mathcal L}^*(-S)\otimes{\mathcal O}_X(S)\, \xrightarrow{\,\,\,\alpha\otimes
{\rm Id}_{{\mathcal O}_X(S)}\,\,\,}\, \widetilde{E}(-S)\otimes{\mathcal O}_X(S)\,=\, \widetilde{E}
$$
that splits holomorphically the short exact sequence in \eqref{e3}.
But, as noted earlier, the short exact sequence in \eqref{e3} does not split holomorphically.
In view of this contradiction we conclude that there is no decomposition as in \eqref{e11}.
\end{proof}

\begin{remark}\label{re-n}
Regarding Proposition \ref{prop1}(1) it should be clarified that although $E$ in \eqref{e4} is isomorphic to
${\mathcal L}\oplus {\mathcal L}^*(-S)$, there is no natural isomorphism between them. Indeed, any two holomorphic
splittings of the short exact sequence \eqref{e4} differ by an element of
$$
H^0(X, \, {\rm Hom}({\mathcal L}^*(-S),\, {\mathcal L}))\,=\, H^0(X,\, K_X(S)).
$$
A holomorphic splittings of the short exact sequence \eqref{e4} produces an isomorphism of
$E_x$ with ${\mathcal L}_x\oplus {\mathcal L}^*(-S)_x$ for any $x\,\in\, X$, but this isomorphism depends on
the choice of the splitting. This shows that Proposition \ref{prop1}(1) does not imply Lemma \ref{lem1}.
\end{remark}

We recall that a parabolic connection on the parabolic vector bundle $E_*$ in \eqref{e8} is a logarithmic
connection $D_0\, :\, E\, \longrightarrow\, E\otimes K_X(S)$ on $E$, singular over $S$, such that the following conditions hold:
\begin{enumerate}
\item for any $x_i\, \in\, S$ the eigenvalues of the residue $\text{Res}(D_0,\, x_i)$ of $D_0$ at $x_i$
are $\frac{c_i+1}{2c_i+1}$ and $\frac{c_i}{2c_i+1}$ (see \eqref{e2}).

\item The eigenspace in $E_{x_i}$ for the eigenvalue $\frac{c_i+1}{2c_i+1}$ of $\text{Res}(D_0,\, x_i)$ 
is the line $${\mathcal L}^*(-S)_x\, \subset\, E_{x_i}$$ in Lemma \ref{lem1}.
\end{enumerate}

Let $D_0\,:\, E\,\longrightarrow\, E\otimes K_X(S)$ be a logarithmic connection on $E$. Take the holomorphic
line subbundle ${\mathcal L}\, \subset\, E$ in \eqref{e4}, and consider the composition of homomorphisms
$$
{\mathcal L}\, \hookrightarrow\, E \, \xrightarrow{\,\,\, D_0\,\,} \, E\otimes K_X(S)\,
\xrightarrow{\,\,\, p\otimes {\rm Id}_{K_X(S)}\,\,} {\mathcal L}^*(-S)\otimes K_X(S)\,=\, {\mathcal L}\, ,
$$
where $p$ is the projection in \eqref{e4}; this composition of homomorphisms will be denoted by ${\mathcal S}(D_0,\,
{\mathcal L})$. This homomorphism
\begin{equation}\label{e14}
{\mathcal S}(D_0,\, {\mathcal L})\, :\, {\mathcal L}\, \longrightarrow\, {\mathcal L}
\end{equation}
is called the second fundamental form of the subbundle $\mathcal L\, \subset\, E$ for the
logarithmic connection $D_0$.
We note that ${\mathcal S}(D_0,\, {\mathcal L})$ is a constant scalar multiplication.

A parabolic connection on $E_*$ induces a holomorphic connection on 
$\det E_* \,=\,{\mathcal O}_X$ (see \eqref{e10}). Note that any holomorphic connection on
${\mathcal O}_X$ is of the form $d+\omega$, where $d$ denotes the de Rham differential and
$\omega\, \in\, H^0(X,\, K_X)$. A parabolic connection $D_0$ on $E_*$ is called
a parabolic $\text{SL}(2,{\mathbb C})$--connection if the connection on $\det E_* \,=\,{\mathcal O}_X$
induced by $D_0$ coincides with the trivial connection $d$.

\begin{corollary}\label{cor1}\mbox{}
\begin{enumerate}
\item The parabolic vector bundle $E_*$ in \eqref{e8} admits a parabolic ${\rm SL}(2,{\mathbb C})$--connection.

\item For any parabolic connection $D_0$ on $E_*$, the second fundamental form 
${\mathcal S}(D_0,\, {\mathcal L})$ in \eqref{e14} is an isomorphism of $\mathcal L$.

\item For any parabolic connection $D_0$ on $E_*$ the local monodromy of $D_0$ around any point of $S$ is
semisimple.
\end{enumerate}
\end{corollary}

\begin{proof}
In view of Remark \ref{re-cc}, from \eqref{e9} and the second statement in Proposition \ref{prop1} it follows
immediately that $E_*$ admits a parabolic
connection. Take a parabolic connection $D_0$ on $E_*$. Let $d+\omega$ be the connection on $\det E_* \,=\,
{\mathcal O}_X$ induced by $D_0$, where $\omega\, \in\, H^0(X,\, K_X)$ and $d$ is the de Rham differential. Then
$D_0- \frac{1}{2}\omega\otimes {\rm Id}_E$ is a parabolic $\text{SL}(2,{\mathbb C})$--connection on $E_*$.

For any parabolic connection $D_0$ on $E_*$, consider the second fundamental form ${\mathcal S}(D_0,\, {\mathcal L})$
in the second statement. If ${\mathcal S}(D_0,\, {\mathcal L})\,=\, 0$, then $D_0$ produces a parabolic connection
on the line subbundle ${\mathcal L}\, \subset\, E$ in \eqref{e4} equipped with the parabolic structure induced by $E_*$.
But the parabolic degree of this parabolic line bundle is $$g-1+\sum_{i=1}^n\frac{c_i}{2c_i+1}\, >\, 0.$$ This implies
that this parabolic line bundle does not admit any parabolic connection. Hence we conclude that
${\mathcal S}(D_0,\, {\mathcal L})\,\not=\, 0$. This implies that ${\mathcal S}(D_0,\, {\mathcal L})$ is an
isomorphism of $\mathcal L$.

The local monodromy of $D_0$ around
any $x\, \in\, S$ is conjugate to $\exp\left(-2\pi\sqrt{-1}\cdot {\rm Res}(D_0,\, x)\right)$
(see Remark \ref{rem1}).
Hence the eigenvalues of the local monodromy for $D_0$ around each $x_i\,\in\, S$ are $\exp\left(-2\pi\sqrt{-1}
\frac{c_i+1}{2c_i+1}\right)$ and $\exp\left(-2\pi\sqrt{-1}\frac{c_i}{2c_i+1}\right)$. This proves the third statement.
\end{proof}

We will see in Corollary \ref{cor5} that the endomorphism ${\mathcal S}(D_0,\, {\mathcal L})$ in Corollary
\ref{cor1}(2) is actually independent of the parabolic connection $D_0$ on $E_*$.

\begin{corollary}\label{cor2}
Take any parabolic connection $D_0$ on $E_*$. There is no holomorphic line subbundle of $E$
preserved by $D_0$.
\end{corollary}

\begin{proof}
Let $L\, \subset\, E$ be a holomorphic line subbundle preserved by $D_0$. Denoted by $L_*$ the parabolic
line bundle defined by the parabolic structure on $L$ induced by $E_*$. Since $D_0$ is a parabolic connection
on $E_*$, its restriction to $L$ is a parabolic connection on $L_*$. Therefore, we have
\begin{equation}\label{d1}
\text{par-deg}(L_*)\,=\, 0.
\end{equation}
Consider the parabolic structure on the quotient ${\mathcal L}^*(-S)$ in \eqref{e4} induced by $E_*$. Its parabolic
degree is negative, and hence from \eqref{d1} we conclude that there is no nonzero parabolic homomorphism from
$L_*$ to it. Consequently, the subbundle $L\, \subset\, E$ coincides with the subbundle $\mathcal L$ in \eqref{e4}.
Since $L\,=\, \mathcal L$ is preserved by $D_0$, the second fundamental form ${\mathcal S}(D_0,\, {\mathcal L})$
in \eqref{e14} vanishes identically. But this contradicts Corollary \ref{cor1}(2). Hence $D_0$ does not
preserve any holomorphic line subbundle of $E$.
\end{proof}

Given a parabolic connection $D$ on $E_*$, consider its monodromy representation
$$
{\rm Mon}_D\, :\, \pi_1(X\setminus S,\, y)\, \longrightarrow\, \text{GL}(2, {\mathbb C})\, ,
$$
where $y\, \in\, X\setminus D$ is a base point. Corollary \ref{cor2} implies that ${\rm Mon}_D$ is irreducible,
meaning the action of ${\rm Mon}_D(\pi_1(X\setminus S,\, y))\, \subset\, \text{GL}(2, {\mathbb C})$ on
${\mathbb C}^2$ does not preserve any line.

\subsection{Orbifold structure}

In this subsection we assume that $\{c_i\}_{i=1}^n$ in \eqref{e2} are all integers; recall that
$c_i\, >\,1$ for all $1\, \leq\, i\, \leq\, n$,

There is a ramified Galois covering
\begin{equation}\label{d2}
\varphi\, :\, Y\, \longrightarrow\, X
\end{equation}
satisfying the following two conditions:
\begin{itemize}
\item $\varphi$ is unramified over the complement $X\setminus S$, and

\item for every $x_i\, \in\, S$ and one (hence every) point $y\, \in\, \varphi^{-1}(x_i)$,
the order of the ramification of $\varphi$ at $y$ is $2c_i+1$.
\end{itemize}
Such a ramified Galois covering $\varphi$ exists; see \cite[p. 26, Proposition 1.2.12]{Na}.
Let
\begin{equation}\label{d3}
\Gamma\,:=\, \text{Gal}(\varphi) \,=\, \text{Aut}(Y/X) \, \subset\, \text{Aut}(Y)
\end{equation}
be the Galois group for the Galois covering $\varphi$. A holomorphic vector bundle $V\, \stackrel{q_0}{\longrightarrow}\, Y$
is called an \textit{orbifold bundle} if $\Gamma$ acts on the total space of $V$ such that
following three conditions hold:
\begin{enumerate}
\item The map $V\, \longrightarrow\, V$ given by the action of any element of $\Gamma$ on $V$ is holomorphic,

\item the projection $q_0$ is $\Gamma$--equivariant, and

\item the action of any $\gamma\, \in\, \Gamma$ on $V$ is a holomorphic automorphism
of the vector bundle $V$ over the automorphism $\gamma$ of $Y$.
\end{enumerate}

Recall that the parabolic weights of $E_*$ at any $x_i\, \in\, S$ are integral multiples of
$\frac{1}{2c_i+1}$. Therefore, there is a unique, up to an isomorphism, orbifold vector
bundle $\mathcal V$ of rank two on $Y$ which corresponds to the parabolic vector bundle $E_*$ \cite{Bi1},
\cite{Bo1}, \cite{Bo2}. The action of $\Gamma$ on this $\mathcal V$ produces an action of $\Gamma$
on the direct image $\varphi_*{\mathcal V}$. We have
\begin{equation}\label{e16}
(\varphi_*{\mathcal V})^\Gamma\,=\, E\, .
\end{equation}
From \eqref{e10} it follows that
\begin{equation}\label{e15}
\det {\mathcal V}\,=\, \bigwedge\nolimits^2 {\mathcal V}\,=\, {\mathcal O}_Y\, ,
\end{equation}
and the action of $\Gamma$ on the orbifold bundle $\det {\mathcal V}$ coincides with the action of
$\Gamma$ on ${\mathcal O}_Y$ given
by the action of $\Gamma$ on $Y$. Consider the subbundle $\mathcal L\, \subset\, E$ in \eqref{e4}. Let
\begin{equation}\label{e17}
\mathbf{L}\, \subset\, \mathcal V
\end{equation}
be the orbifold line subbundle corresponding to it. So the action of $\Gamma$ on
$\mathcal V$ preserves the subbundle $\mathbf{L}$, and the subbundle
$$
(\varphi_*{\mathbf L})^\Gamma\,\subset\, (\varphi_*{\mathcal V})^\Gamma\,=\, E\, .
$$
coincides with $\mathcal L$.

The action of $\Gamma$ on $Y$ produces an action of $\Gamma$ on the canonical bundle $K_Y$.
For any automorphism $\gamma\,\in\, \Gamma$ consider its differential $d\gamma\, :\,
TY\, \longrightarrow\, \gamma^*TY$. The action of $\gamma$ on $K_Y$ is given by
$((d\gamma)^*)^{-1}\,=\, (d\gamma^{-1})^*$. Therefore, $K_Y$ is an orbifold line bundle.

\begin{lemma}\label{lem2}
The orbifold line bundle $\mathbf{L}^{\otimes 2}$ (see \eqref{e17}) is isomorphic to the orbifold
line bundle $K_Y$.
\end{lemma}

\begin{proof}
Let ${\mathcal L}_*$ denote the holomorphic line subbundle $\mathcal L$ in \eqref{e4} equipped with
the parabolic structure on it induced by $E_*$. So the underlying holomorphic line
bundle for the parabolic bundle ${\mathcal L}_*\otimes{\mathcal L}_*$ is $K_X$, and the
parabolic weight at any $x_i\, \in\, S$ is $\frac{2c_i}{2c_i+1}$. Hence the orbifold
line bundle on $Y$ corresponding to ${\mathcal L}_*\otimes{\mathcal L}_*$ is
$$
(\varphi^*K_X)\otimes {\mathcal O}_Y\left(\sum_{i=1}^n 2c_i\varphi^{-1}(x_i)_{\rm red}\right)\,=\, K_Y
$$
equipped with the action of $\Gamma$ given by the action $\Gamma$ on $Y$,
where $\varphi^{-1}(x_i)_{\rm red}$ is the reduced inverse image of $x_i$.
Since the orbifold line bundle $\mathbf{L}^{\otimes 2}$ corresponds to the parabolic
line bundle ${\mathcal L}_*\otimes{\mathcal L}_*$, the lemma follows.
\end{proof}

{}From Lemma \ref{lem2} it follows that $\mathbf L$ is an orbifold theta characteristic
on $Y$, and from \eqref{e15} we have a short exact sequence of orbifold bundles
\begin{equation}\label{e18}
0\, \longrightarrow\, {\mathbf L} \, \longrightarrow\, \mathcal{V} 
\, \longrightarrow\, {\mathbf L}^* \, \longrightarrow\, 0\, .
\end{equation}

\begin{corollary}\label{cor3}
The short exact sequence in \eqref{e18} does not admit any $\Gamma$--equivariant
holomorphic splitting.
\end{corollary}

\begin{proof}
If \eqref{e18} has a $\Gamma$--equivariant holomorphic splitting, then $\mathcal V$
is a direct sum of orbifold line bundles. This would imply that the parabolic vector
bundle $E_*$ --- that corresponds to $\mathcal V$ --- is a direct sum of parabolic line bundles.
Therefore, from Proposition \ref{prop1}(2) it follows that
\eqref{e18} does not admit any $\Gamma$--equivariant holomorphic splitting.
\end{proof}

Actually a stronger form of Corollary \ref{cor3} can be proved using it.

\begin{proposition}\label{prop2}
The short exact sequence of holomorphic vector bundles in \eqref{e18} does not admit any
holomorphic splitting.
\end{proposition}

\begin{proof}
Assume that there is a holomorphic splitting
$$\rho\,\, :\,\, {\mathbf L}^* \, \longrightarrow\, \mathcal{V}$$
of the short exact sequence of holomorphic vector bundles in \eqref{e18}.
Although $\rho$ itself may not be $\Gamma$--equivariant, using it we will construct a
$\Gamma$--equivariant splitting. For any $\gamma\, \in\,\Gamma$, the composition of
homomorphisms
$$
{\mathbf L}^* \, \stackrel{\gamma}{\longrightarrow}\, {\mathbf L}^* \, \stackrel{\rho}{\longrightarrow}\, 
{\mathcal V} \, \xrightarrow{\,\,\gamma^{-1}\,\,}\, \mathcal V\, ,
$$
which will be denoted by $\rho[\gamma]$, is also a holomorphic splitting of
the short exact sequence of holomorphic vector bundles in \eqref{e18}. Now the average
$$
\widetilde{\rho}\,:=\, \frac{1}{\# \Gamma}\sum_{\gamma\in \Gamma}\rho[\gamma]
\,\, :\,\, {\mathbf L}^* \, \longrightarrow\, \mathcal{V},
$$
where $\#\Gamma$ is the order of $\Gamma$, is a $\Gamma$--equivariant holomorphic splitting
of the short exact sequence of holomorphic vector bundles in \eqref{e18}. But this contradicts
Corollary \ref{cor3}. Therefore, the short exact sequence of holomorphic vector bundles
in \eqref{e18} does not admit any holomorphic splitting.
\end{proof}

The $\Gamma$--invariant holomorphic connections on $\mathcal V$ correspond to the
parabolic connections on $E_*$. Moreover, the parabolic ${\rm SL}(2,{\mathbb C})$--connections
on $E_*$ correspond to the $\Gamma$--invariant holomorphic connections $D_V$ on $\mathcal V$ 
that satisfy the condition that the holomorphic connection on 
$\det {\mathcal V}\,=\, {\mathcal O}_Y$ (see \eqref{e15}) induced by $D_V$ is the
trivial connection on ${\mathcal O}_Y$ given by the de Rham differential.

\begin{lemma}\label{lem3}
The orbifold vector bundle $\mathcal V$ admits ${\rm SL}(2,{\mathbb C})$--oper connections.
The parabolic ${\rm SL}(2,{\mathbb C})$--connections on the parabolic bundle $E_*$ are precisely the
$\Gamma$--invariant ${\rm SL}(2,{\mathbb C})$--oper structures on the orbifold bundle $\mathcal V$.
\end{lemma}

\begin{proof}
{}From Proposition \ref{prop2} it follows immediately that $\mathcal V$ admits ${\rm SL}(2,{\mathbb C})$--oper connections.
Now the second statement of the lemma is deduced from the above observation that
the parabolic ${\rm SL}(2,{\mathbb C})$--connections
on $E_*$ correspond to the $\Gamma$--invariant holomorphic connections $D_V$ on $\mathcal V$ 
that satisfy the condition that the holomorphic connection on 
$\det {\mathcal V}\,=\, {\mathcal O}_Y$ induced by $D_V$ is the
trivial connection on ${\mathcal O}_Y$.
\end{proof}

\section{Symmetric powers of parabolic bundle}\label{se3}

\subsection{Explicit description of some symmetric powers}

In Section \ref{se3.2} we will define parabolic ${\rm SL}(r,{\mathbb C})$--opers for all
$r\, \geq\, 2$. The definition involves symmetric powers of the parabolic vector bundle $E_*$
in \eqref{e9}. Keeping this in mind, we will explicitly describe a few low degree symmetric powers of the parabolic vector bundle $E_*$.
This will done using the alternative description of parabolic bundles --- given by Maruyama and Yokogawa
in \cite{MY} (see also \cite{Yo} and \cite[Appendix A3]{BDHP}) --- as filtered sheaves. This approach of \cite{MY}
is better suited for handling the tensor product, symmetric product exterior product of parabolic vector bundles.

First we will describe the second symmetric power $\text{Sym}^2(E_*)$ of the
parabolic vector bundle $E_*$. Consider the rank three holomorphic vector bundle 
$\text{Sym}^2(E)$, where $E$ is the vector bundle in \eqref{e4}. Since $\text{Sym}^2(E)$ is a quotient
of $E^{\otimes 2}$, any subspace of $E^{\otimes 2}_x$ produces a subspace of $\text{Sym}^2(E)_x$.
For each $x_i\, \in\, S$, let
$$B_i\, \subset\, \text{Sym}^2(E)_{x_i}\,=\, \text{Sym}^2(E_{x_i})$$ be the subspace given by the image of
$$E_{x_i}\otimes {\mathcal L}^*(-S)_{x_i}\,\subset\, E^{\otimes 2}_{x_i}$$ in $\text{Sym}^2
(E_{x_i})$, where ${\mathcal L}^*(-S)_{x_i}\,\subset\, E_{x_i}$ is the subspace in Lemma \ref{lem1}.
Consider the unique holomorphic vector bundle $E^2$ of rank three on $X$ that fits in the
following short exact sequence of sheaves
\begin{equation}\label{e19}
0\,\longrightarrow\, E^2\, \longrightarrow\,{\rm Sym}^2(E)(S)\,:=\, {\rm Sym}^2(E)\otimes{\mathcal O}_X(S)
\end{equation}
$$
\longrightarrow
\,\bigoplus_{i=1}^n \left(\text{Sym}^2(E)_{x_i}/B_i\right)\otimes{\mathcal O}_X(S)_{x_i}\, \longrightarrow\, 0\, .
$$
The holomorphic vector bundle underlying the parabolic vector bundle $\text{Sym}^2(E_*)$ is $E^2$.

\begin{lemma}\label{lem4}
For every $x_i\, \in\, S$, the fiber $E^2_x$ fits in a natural exact sequence
$$
0\, \longrightarrow\, {\mathcal L}^{\otimes 2}_{x_i}\, \longrightarrow\, E^2_{x_i}\,
\longrightarrow\, B_i\otimes{\mathcal O}_X(S)_{x_i}
$$
$$
=\, (E_{x_i}\otimes {\mathcal L}^*(-S)_{x_i})
\otimes{\mathcal O}_X(S)_{x_i} \,=\, (E\otimes {\mathcal L}^*)_{x_i} \, \longrightarrow\, 0\, .
$$
\end{lemma}

\begin{proof}
Consider the commutative digram
$$
\begin{matrix}
0 & \longrightarrow & \text{Sym}^2(E) & \longrightarrow & {\rm Sym}^2(E)(S)& \longrightarrow &
\bigoplus_{i=1}^n \text{Sym}^2(E)(S)_{x_i}& \longrightarrow & 0\\
&&\,\, \Big\downarrow \mathbf{f}&& \,\,\,\,\,\Big\downarrow {\rm Id}&& \Big\downarrow\\
0& \longrightarrow & E^2 & \longrightarrow & {\rm Sym}^2(E)(S) & \longrightarrow &
\bigoplus_{i=1}^n \frac{\text{Sym}^2(E)(S)_{x_i}}{B_i\otimes{\mathcal O}_X(S)_{x_i}} & \longrightarrow & 0.
\end{matrix}
$$
For any $x\, \in\, S$, the map ${\mathbf f}(x)\, :\, \text{Sym}^2(E)_x\, \longrightarrow\,
E^2_x$ is injective on the subspace ${\mathcal L}^{\otimes 2}_{x_i}\, \hookrightarrow\, \text{Sym}^2(E)_{x_i}$, and
moreover ${\mathbf f}(x_i)({\mathcal L}^{\otimes 2}_{x_i})\, \subset\, E^2_{x_i}$
coincides with ${\mathbf f}(x_i)(\text{Sym}^2(E)_{x_i})$. Therefore, the subspace
${\mathcal L}^{\otimes 2}_{x_i}\, \hookrightarrow\, E^2_{x_i}$ in the lemma is the
image of the homomorphism ${\mathbf f}(x_i)$.

For the map $E^2\, \longrightarrow\, {\rm Sym}^2(E)(S)\, :=\, {\rm Sym}^2(E)\otimes{\mathcal O}_X(S)$ in
\eqref{e19}, the image of $E^2_{x_i}$ is
$$
B_i\otimes{\mathcal O}_X(S)_{x_i}
\,=\, (E_{x_i}\otimes {\mathcal L}^*(-S)_{x_i})
\otimes{\mathcal O}_X(S)_{x_i} \,=\, (E\otimes {\mathcal L}^*)_{x_i}\, \subset\, 
{\rm Sym}^2(E)(S)_{x_i}\, .
$$
This proves the lemma.
\end{proof}

For any $x_i\, \in\, S$, consider the subspace 
$$
{\mathcal L}^*(-S)^{\otimes 2}_{x_i}\,\subset\, B_i\,=\, ({\mathcal L}_{x_i}\otimes {\mathcal L}^*(-S)_{x_i})
\oplus {\mathcal L}^*(-S)^{\otimes 2}_{x_i}\, .
$$
Let
\begin{equation}\label{e20}
\mathcal{F}_i\, \subset\, E^2_{x_i}
\end{equation}
be the inverse image of ${\mathcal L}^*(-S)^{\otimes 2}_{x_i}\otimes {\mathcal O}_X(S)_{x_i}\,\subset\,
B_i\otimes {\mathcal O}_X(S)_{x_i}$ for the quotient map $E^2_{x_i}\,
\longrightarrow\, B_i\otimes{\mathcal O}_X(S)_{x_i}$ in Lemma \ref{lem4}. 

As mentioned before, the holomorphic vector bundle underlying the parabolic vector bundle $\text{Sym}^2(E_*)$ is $E^2$.
The quasiparabolic filtration of $E^2_{x_i}$, where $x_i\, \in\, S$, is the following:
\begin{equation}\label{e21}
{\mathcal L}^{\otimes 2}_{x_i}\,\, \subset\, \,\mathcal{F}_i\,\, \subset\, \,E^2_{x_i}\, ,
\end{equation}
where ${\mathcal L}^{\otimes 2}_{x_i}$ and ${\mathcal F}_i$ are the subspaces in Lemma \ref{lem4} and \eqref{e20}
respectively. The parabolic weight of ${\mathcal L}^{\otimes 2}_{x_i}$ is $\frac{2c_i}{2c_i+1}$
and the parabolic weight of $\mathcal{F}_i$ is $\frac{1}{2c_i+1}$; the parabolic weight of $E^2_{x_i}$ is $0$.

The parabolic symmetric product $\text{Sym}^3(E_*)$ is actually a little easier to describe. The holomorphic
vector bundle underlying the parabolic vector bundle $\text{Sym}^3(E_*)$ is the rank four vector bundle
\begin{equation}\label{e22}
E^3\,:=\, (\text{Sym}^3(E))\otimes {\mathcal O}_X(S).
\end{equation}
For each $x_i\, \in\, S$, the decomposition of $E_{x_i}$ in Lemma \ref{lem1} gives the
following decomposition of the fiber $E^3_{x_i}$:
\begin{equation}\label{e24}
\left(({\mathcal L}^*(-S)^{\otimes 3}_{x_i})\oplus ({\mathcal L}^*(-S)^{\otimes 2}_{x_i}\otimes {\mathcal L}_{x_i})
\oplus ({\mathcal L}^*(-S)_{x_i} \otimes {\mathcal L}^{\otimes 2}_{x_i})
\oplus ({\mathcal L}^{\otimes 3}_{x_i})\right)\otimes {\mathcal O}_X(S)_{x_i}\,=\, E^3_{x_i}\, .
\end{equation}
The quasiparabolic filtration of $E^3_{x_i}$ is
\begin{equation}\label{e23}
({\mathcal L}^*(-S)^{\otimes 3}_{x_i})\otimes {\mathcal O}_X(S)_{x_i}\, \subset\,
\left(({\mathcal L}^*(-S)^{\otimes 3}_{x_i})\oplus ({\mathcal L}^*(-S)^{\otimes 2}_{x_i}\otimes {\mathcal L}_{x_i})\right)
\otimes {\mathcal O}_X(S)_{x_i}
\end{equation}
$$
\subset\, \left(({\mathcal L}^*(-S)^{\otimes 3}_{x_i})\oplus ({\mathcal L}^*(-S)^{\otimes 2}_{x_i}\otimes {\mathcal L}_{x_i})
\oplus ({\mathcal L}^*(-S)_{x_i} \otimes {\mathcal L}^{\otimes 2}_{x_i})\right)\otimes {\mathcal O}_X(S)_{x_i}\,\subset\, E^3_{x_i}.
$$
The parabolic weight of ${\mathcal L}^*(-S)^{\otimes 3}_{x_i}\otimes {\mathcal O}_X(S)_{x_i}$ is $\frac{c_i+2}{2c_i+1}$,
The parabolic weight of $$\left(({\mathcal L}^*(-S)^{\otimes 3}_{x_i})\oplus ({\mathcal L}^*(-S)^{\otimes 2}_{x_i}\otimes {\mathcal L}_{x_i})
\right)\otimes {\mathcal O}_X(S)_{x_i}$$
is $\frac{c_i+1}{2c_i+1}$, the parabolic weight of $\left(({\mathcal L}^*(-S)^{\otimes 3}_{x_i})\oplus ({\mathcal L}^*(-S)^{\otimes 2}_{x_i}
\otimes {\mathcal L}_{x_i})\oplus ({\mathcal L}^*(-S)_{x_i} \otimes {\mathcal L}^{\otimes 2}_{x_i})\right)\otimes {\mathcal O}_X(S)_{x_i}$ is
$\frac{c_i}{2c_i+1}$, and the parabolic weight of $E^3_{x_i}$ is $\frac{c_i-1}{2c_i+1}$.

Finally, we will describe the parabolic symmetric product $\text{Sym}^4(E_*)$. Consider the rank five vector bundle
$$
\text{Sym}^4(E)(2S)\,=\, (\text{Sym}^4(E))\otimes {\mathcal O}_X(2S)\, .
$$
Using Lemma \ref{lem1}, the fiber $\text{Sym}^4(E)(2S)_{x_i}$, where $x_i\, \in\, S$, decomposes
into a direct sum of lines. More precisely, as in \eqref{e24},
\begin{equation}\label{e27}
\text{Sym}^4(E)(2S)_{x_i}\,=\, (({\mathcal L}^*)^{\otimes 4}(-2S))_{x_i}\oplus
(({\mathcal L}^*)^{\otimes 3}\otimes {\mathcal L}(-S))_{x_i}
\end{equation}
$$
\oplus (({\mathcal L}^*)^{\otimes 2}\otimes {\mathcal L}^{\otimes 2})_{x_i}
\oplus ({\mathcal L}^* \otimes {\mathcal L}^{\otimes 3}(S))_{x_i}\oplus ({\mathcal L}^{\otimes 4}(2S))_{x_i}.
$$
Let $E^4$ denote the vector bundle of rank five defined by the following short exact sequence of sheaves:
\begin{equation}\label{e25}
0\, \longrightarrow\, E^4 \, \stackrel{\mathbf h}{\longrightarrow}\, \text{Sym}^4(E)(2S)\, \longrightarrow\,
\end{equation}
$$
\bigoplus_{i=1}^n {\mathcal Q}_i\, =\,
\bigoplus_{i=1}^n \frac{\text{Sym}^4(E)(2S)_{x_i}}{(({\mathcal L}^*)^{\otimes 4}(-2S))_{x_i}\oplus
(({\mathcal L}^*)^{\otimes 3}\otimes {\mathcal L}(-S))_{x_i}\oplus
(({\mathcal L}^*)^{\otimes 2}\otimes {\mathcal L}^{\otimes 2})_{x_i}}\, \longrightarrow\, 0,
$$
where 
\begin{equation}\label{e25a}
{\mathcal Q}_i\, :=\,
\frac{\text{Sym}^4(E)(2S)_{x_i}}{(({\mathcal L}^*)^{\otimes 4}(-2S))_{x_i}\oplus
(({\mathcal L}^*)^{\otimes 3}\otimes {\mathcal L}(-S))_{x_i}\oplus
(({\mathcal L}^*)^{\otimes 2}\otimes {\mathcal L}^{\otimes 2})_{x_i}}.
\end{equation}
The holomorphic vector bundle underlying the parabolic vector bundle $\text{Sym}^4(E_*)$ is $E^4$ defined
in \eqref{e25}.

\begin{lemma}\label{lem5}
For every $x_i\, \in\, S$, the fiber $E^4_{x_i}$ fits in the following short exact sequence of vector spaces:
$$
0\, \longrightarrow\, ({\mathcal L}^* \otimes {\mathcal L}^{\otimes 3})_{x_i}\oplus ({\mathcal L}^{\otimes 4}(S))_{x_i}
\, \longrightarrow\, E^4_{x_i} 
$$
$$
\xrightarrow{\,\,\rho_i\,\,}\, (({\mathcal L}^*)^{\otimes 4}(-2S))_{x_i}\oplus
(({\mathcal L}^*)^{\otimes 3}\otimes {\mathcal L}(-S))_{x_i}\oplus
(({\mathcal L}^*)^{\otimes 2}\otimes {\mathcal L}^{\otimes 2})_{x_i}\, \longrightarrow\, 0\, .
$$
\end{lemma}

\begin{proof}
The projection
$$
\rho_i\, :\, E^4_{x_i} \, \longrightarrow\, (({\mathcal L}^*)^{\otimes 4}(-2S))_{x_i}\oplus
(({\mathcal L}^*)^{\otimes 3}\otimes {\mathcal L}(-S))_{x_i}\oplus
(({\mathcal L}^*)^{\otimes 2}\otimes {\mathcal L}^{\otimes 2})_{x_i}
$$
in the lemma is given by the homomorphism ${\mathbf h}(x_i)$ in \eqref{e25}. To describe the homomorphism
$$
({\mathcal L}^* \otimes {\mathcal L}^{\otimes 3})_{x_i}\oplus ({\mathcal L}^{\otimes 4}(S))_{x_i}\, \longrightarrow\, E^4_{x_i} 
$$
in the lemma, we consider the commutative diagram of homomorphisms
$$
\begin{matrix}
0 & \longrightarrow & \text{Sym}^4(E)(S) & \longrightarrow & {\rm Sym}^4(E)(2S)& \longrightarrow &
\bigoplus_{i=1}^n \text{Sym}^4(E)(2S)_{x_i}& \longrightarrow & 0\\
&&\,\, \Big\downarrow \mathbf{f}&& \Big\downarrow && \Big\downarrow\\
0& \longrightarrow & E^4 & \longrightarrow & {\rm Sym}^4(E)(2S) & \longrightarrow &
\bigoplus_{i=1}^n {\mathcal Q}_i& \longrightarrow & 0
\end{matrix}
$$
where ${\mathcal Q}_i$ is defined in \eqref{e25a}. Let
\begin{equation}\label{e26}
{\mathbf f}(x_i)\,:\, \text{Sym}^4(E)(S)_{x_i}\, \longrightarrow\, E^4_{x_i}
\end{equation}
be the restriction of it to $x_i\, \in\, S$. As in \eqref{e27}, we have the decomposition
$$
\text{Sym}^4(E)(S)_{x_i}\,=\, (({\mathcal L}^*)^{\otimes 4}(-3S))_{x_i}\oplus
(({\mathcal L}^*)^{\otimes 3}\otimes {\mathcal L}(-2S))_{x_i}
$$
$$
\oplus (({\mathcal L}^*)^{\otimes 2}\otimes {\mathcal L}^{\otimes 2}(-S))_{x_i}
\oplus ({\mathcal L}^* \otimes {\mathcal L}^{\otimes 3})_{x_i}\oplus ({\mathcal L}^{\otimes 4}(S))_{x_i}.
$$
The subspace
$$
(({\mathcal L}^*)^{\otimes 4}(-3S))_{x_i}\oplus
(({\mathcal L}^*)^{\otimes 3}\otimes {\mathcal L}(-2S))_{x_i}
\oplus (({\mathcal L}^*)^{\otimes 2}\otimes {\mathcal L}^{\otimes 2}(-S))_{x_i}
\, \subset\, \text{Sym}^4(E)(S)_{x_i}
$$
is the kernel of the homomorphism ${\mathbf f}(x_i)$ in \eqref{e26}. The restriction of
${\mathbf f}(x_i)$ to the subspace
$$
({\mathcal L}^* \otimes {\mathcal L}^{\otimes 3})_{x_i}\oplus ({\mathcal L}^{\otimes 4}(S))_{x_i}
\, \subset\, \text{Sym}^4(E)(S)_{x_i}
$$
is injective. Therefore, ${\mathbf f}(x_i)$ gives the homomorphism
$$
({\mathcal L}^* \otimes {\mathcal L}^{\otimes 3})_{x_i}\oplus ({\mathcal L}^{\otimes 4}(S))_{x_i}
\, \longrightarrow\, E^4_{x_i} 
$$
in the lemma. It is evident that the quotient map $E^4_{x_i} \, \longrightarrow\, E^4_{x_i}/( 
({\mathcal L}^* \otimes {\mathcal L}^{\otimes 3})_{x_i}\oplus ({\mathcal L}^{\otimes 4}(S))_{x_i})$ coincides
with $\rho_i$.
\end{proof}

Define the subspaces
\begin{equation}\label{e28}
{\mathcal F}^i_3 \, :=\, \rho^{-1}_i((({\mathcal L}^*)^{\otimes 4}(-2S))_{x_i})\, \subset\,
{\mathcal F}^i_4 \, :=\, \rho^{-1}_i((({\mathcal L}^*)^{\otimes 4}(-2S))_{x_i}\oplus
(({\mathcal L}^*)^{\otimes 3}\otimes {\mathcal L}(-S))_{x_i})\, \subset\, E^4_{x_i}
\end{equation}
where $\rho_i$ is the homomorphism in Lemma \ref{lem5}.

As mentioned before, the holomorphic vector bundle underlying the parabolic vector bundle $\text{Sym}^4(E_*)$ is $E^4$.
The quasiparabolic filtration of $E^4_{x_i}$ is
$$
({\mathcal L}^* \otimes {\mathcal L}^{\otimes 3})_{x_i}\,\subset\, 
({\mathcal L}^* \otimes {\mathcal L}^{\otimes 3})_{x_i}\oplus ({\mathcal L}^{\otimes 4}(S))_{x_i}
\, \subset\,{\mathcal F}^i_3 \, \subset\,{\mathcal F}^i_4 \, \subset\, E^4_{x_i}
$$
(see Lemma \ref{lem5} and \eqref{e28}). The parabolic weight of $({\mathcal L}^* \otimes {\mathcal L}^{\otimes 3})_{x_i}$
is $\frac{2c_i}{2c_i+1}$, the parabolic weight of $({\mathcal L}^* \otimes {\mathcal L}^{\otimes 3})_{x_i}
\oplus ({\mathcal L}^{\otimes 4}(S))_{x_i}$ is $\frac{2c_i-1}{2c_i+1}$, the parabolic weight
of ${\mathcal F}^i_3$ is $\frac{2}{2c_i+1}$, the parabolic weight of ${\mathcal F}^i_4$ is $\frac{1}{2c_i+1}$ and
the parabolic weight of $E^4_{x_i}$ is $0$.

\subsection{Higher rank parabolic opers}\label{se3.2}

For any $r\, \geq\, 2$, consider the parabolic vector bundle of rank $r$ defined by the symmetric product $\text{Sym}^{r-1}(E_*)$
of the parabolic vector bundle $E_*$ in \eqref{e8}. Since $\det E_*\,=\, {\mathcal O}_X$ (see \eqref{e10}), it follows
that
\begin{equation}\label{ea1}
\det \text{Sym}^{r-1}(E_*)\,=\, \bigwedge\nolimits^r \text{Sym}^{r-1}(E_*)\,=\, {\mathcal O}_X,
\end{equation} 
where ${\mathcal O}_X$ is equipped with the trivial parabolic structure (no nonzero parabolic weights).

A parabolic
${\rm SL}(r,{\mathbb C})$--connection on $\text{Sym}^{r-1}(E_*)$ is a parabolic connection on $\text{Sym}^{r-1}(E_*)$ satisfying the
condition that the induced parabolic connection on $\det \text{Sym}^{r-1}(E_*)\,=\, {\mathcal O}_X$ is the trivial connection.

Two parabolic ${\rm SL}(r,{\mathbb C})$--connections on $\text{Sym}^{r-1}(E_*)$ are called equivalent if they
differ by a holomorphic automorphism of the parabolic bundle $\text{Sym}^{r-1}(E_*)$. If $D_1$ is
a parabolic ${\rm SL}(r,{\mathbb C})$--connection on $\text{Sym}^{r-1}(E_*)$, and $D_2$ is another
parabolic connection on $\text{Sym}^{r-1}(E_*)$ equivalent to $D_1$, then $D_2$ is clearly
a parabolic ${\rm SL}(r,{\mathbb C})$--connection. Indeed, this follows immediately from the fact that the
holomorphic automorphisms of a holomorphic line bundle $\mathbb L$ on $X$ act trivially on the space of all
logarithmic connections on $\mathbb L$.

\begin{definition}\label{def1}
A parabolic $\text{SL}(r,{\mathbb C})$--\textit{oper} on $X$ is an equivalence class of parabolic
${\rm SL}(r,{\mathbb C})$--connections on $\text{Sym}^{r-1}(E_*)$.
\end{definition}

\begin{remark}\label{re-n2}
It should be clarified that the class of parabolic $\text{SL}(r,{\mathbb C})$--opers in Definition \ref{def1} is different from the
class in \cite{BDP} (see \cite[p.~504, Definition 4.1]{BDP} and \cite[p.~511, Definition 5.2]{BDP}). Indeed, the parabolic
vector bundle $E_*$ in \eqref{e9} is different from the one in \cite{BDP} (see \cite[p.~497, (3.4)]{BDP}, \cite[p.~497, (3.5)]{BDP}).
In fact the underlying rank two bundles are different and the parabolic weights are also different. In the nonparabolic
case there is only one class of $\text{SL}(r,{\mathbb C})$--opers. Roughly speaking, parabolic $\text{SL}(r,{\mathbb C})$--opers
can be considered as equivariant opers and the two classes of parabolic $\text{SL}(r,{\mathbb C})$--opers arise because of two different
types of equivariant structures.
\end{remark}

\begin{proposition}\label{prop3}\mbox{}
\begin{enumerate}
\item The parabolic vector bundle ${\rm Sym}^{r-1}(E_*)$ admits a parabolic ${\rm SL}(r,{\mathbb C})$--connection.

\item For any parabolic connection $D_r$ on ${\rm Sym}^{r-1}(E_*)$, the local monodromy of $D_r$ around any
$x_i\,\in\, S$ is semisimple.
\end{enumerate}
\end{proposition}

\begin{proof}
Any parabolic connection on $E_*$ induces a parabolic connection on $\text{Sym}^{r-1}(E_*)$. Moreover, a
parabolic ${\rm SL}(2,{\mathbb C})$--connection on $E_*$ induces a parabolic ${\rm SL}(r,{\mathbb C})$--connection
on $\text{Sym}^{r-1}(E_*)$. Therefore, from Corollary \ref{cor1}(1) it follows that $\text{Sym}^{r-1}(E_*)$
admits a parabolic connection on $E_*$.

Let $D_2$ be a parabolic ${\rm SL}(2,{\mathbb C})$--connection on $E_*$. Denote by
$D_r$ the parabolic connection on $\text{Sym}^{r-1}(E_*)$ induced by $D_2$. From Corollary 
\ref{cor1}(3) we know that the local monodromy of $D_2$ around any $x_i\,\in\, S$ is 
semisimple. Since the local monodromy of $D_r$ around any $x_i\,\in\, S$ is simply the
$(r-1)$-th symmetric product of the local monodromy of $D_2$ around $x_i\,\in\, S$, and
the local monodromy of $D_2$ around $x_i\,\in\, S$ is semisimple, it follows that
the local monodromy of $D_r$ around $x_i\,\in\, S$ is semisimple.

We have shown that $\text{Sym}^{r-1}(E_*)$ admits a parabolic connection for which the local monodromy around any
$x_i\,\in\, S$ is semisimple. On the other hand, the space of parabolic connections on $\text{Sym}^{r-1}(E_*)$
is an affine space for the vector space $$H^0(X,\, \text{End}^n(\text{Sym}^{r-1}(E_*)) \otimes K_X(S)),$$ where
\begin{equation}\label{ne}
\text{End}^n(\text{Sym}^{r-1}(E_*))\, \subset\, \text{End}(\text{Sym}^{r-1}(E_*))
\end{equation}
is the subsheaf defined by the sheaf of endomorphisms nilpotent with respect to the 
quasiparabolic filtrations of $\text{Sym}^{r-1}(E_*)$ over $S$. Consequently, using Remark \ref{rem1}
it follows that for every parabolic connection $D'_r$ on $\text{Sym}^{r-1}(E_*)$ the local
monodromy of $D'_r$ around any $x_i\,\in\, S$ is semisimple.
\end{proof}

In the rest of this section we assume that $c_i$, $1\, \leq\, i\, \leq\, n$, in \eqref{e2} are integers. 
Take a ramified Galois covering
$\varphi\, :\, Y\, \longrightarrow\, X$ as in \eqref{d2}. As in Section \ref{se2}, let $\mathcal
V$ denote the orbifold bundle on $Y$ corresponding to the parabolic bundle $E_*$ on $X$.
The action of the Galois group
$\Gamma\,=\, \text{Gal}(\varphi)$ on $\mathcal V$ produces an action of $\Gamma$ on $\text{Sym}^{r-1}({\mathcal V})$.
A holomorphic connection on $\text{Sym}^{r-1}({\mathcal V})$ is called \textit{equivariant}
if it is preserved by the action of $\Gamma$ on $\text{Sym}^{r-1}({\mathcal V})$. 

From \eqref{ea1} it follows immediately that $$\det \text{Sym}^{r-1}({\mathcal V})\,=\, \bigwedge\nolimits^r
\text{Sym}^{r-1}({\mathcal V})\,=\, {\mathcal O}_Y .$$
An ${\rm SL}(r,{\mathbb C})$--connection on $\text{Sym}^{r-1}({\mathcal V})$ is a holomorphic
connection $D'_r$ on $\text{Sym}^{r-1}({\mathcal V})$ such that the connection on $\det \text{Sym}^{r-1}({\mathcal V})\,=\,
{\mathcal O}_Y$ induced by $D'_r$ coincides with the trivial connection on ${\mathcal O}_Y$. Two equivariant
${\rm SL}(r,{\mathbb C})$--connections on $\text{Sym}^{r-1}({\mathcal V})$ are called equivalent if they differ
by a holomorphic $\Gamma$--equivariant automorphism of $\text{Sym}^{r-1}({\mathcal V})$.

\begin{proposition}\label{prop4}
There is a natural bijection between the parabolic ${\rm SL}(r,{\mathbb C})$--opers on $X$ and
the equivalence classes of equivariant ${\rm SL}(r,{\mathbb C})$--connections on ${\rm Sym}^{r-1}({\mathcal V})$. 
\end{proposition}

\begin{proof}
Let $D_2$ be a parabolic connection on $E_*$. Since the local monodromy of $D_2$ around any
$x_i\,\in\, S$ is semisimple, it corresponds to an equivariant holomorphic connection $\widehat{D}_2$ on
$\mathcal V$. Let $\widehat{D}_r$ be the equivariant connection on $\text{Sym}^{r-1}({\mathcal V})$ induced by
$\widehat{D}_2$. As before, $D_r$ denotes the parabolic connection on $\text{Sym}^{r-1}(E_*)$ induced by $D_2$.
Therefore, $\widehat{D}_r$ corresponds to $D_r$.

The holomorphic vector bundle underlying the parabolic bundle
$\text{Sym}^{r-1}(E_*)$ is denoted by $\text{Sym}^{r-1}(E_*)_0$ \cite{MY}. As in \eqref{ne}, let
$$\text{End}^n(\text{Sym}^{r-1}(E_*))\, \subset\, \text{End}(\text{Sym}^{r-1}(E_*)_0)$$
be the coherent analytic subsheaf consisting of all locally defined sections $s$ of the
endomorphism bundle $\text{End}(\text{Sym}^{r-1}(E_*)_0)$
satisfying the condition that $s(x)$ is nilpotent with respect to the 
quasi-parabolic filtration of $\text{Sym}^{r-1}(E_*)_x$, for all $x\, \in\, S$ lying in the domain of $s$.
Recall that any parabolic connection on $\text{Sym}^{r-1}(E_*)$ is of the form
$D_r+\theta$ for some $$\theta\, \in\, H^0(X,\, \text{End}^n(\text{Sym}^{r-1}(E_*))\otimes K_X(S)).$$
We have
\begin{equation}\label{e29}
H^0(X,\, \text{End}^n(\text{Sym}^{r-1}(E_*))\otimes K_X(S))\,=\, H^0(Y,\, \text{End}(\text{Sym}^{r-1}({\mathcal V})))^\Gamma .
\end{equation}
Also the space of all equivariant holomorphic connections on ${\rm Sym}^{r-1}({\mathcal V})$ is an affine space for
$H^0(Y,\, \text{End}(\text{Sym}^{r-1}({\mathcal V})))^\Gamma$.

The parabolic connection $D_r+\theta$, where $\theta\, \in\,H^0(X,\, \text{End}^n(\text{Sym}^{r-1}(E_*))\otimes K_X(S))$,
corresponds to the equivariant connection $\widehat{D}_r+\widehat{\theta}$ on
$\text{Sym}^{r-1}({\mathcal V})$, where $\widehat{\theta}\, \in\, H^0(Y,\, \text{End}(\text{Sym}^{r-1}({\mathcal V})))^\Gamma$
corresponds to $\theta$ by the isomorphism in \eqref{e29}.
Also, parabolic automorphisms of $\text{Sym}^{r-1}(E_*)$ are identified with the $\Gamma$--equivariant automorphisms of
$\mathcal V$. Now the proposition follows from \eqref{e29}, Proposition \ref{prop3} and Definition \ref{def1}.
\end{proof}

The above Proposition \ref{prop4} is a generalization of Theorem 6.3 in \cite{BDP} where
a similar statement was proved under the extra assumption that $r$ is odd.

\section{Some properties of parabolic opers} \label{se4}

Consider the vector bundle $E$ in \eqref{e4}. Let
\begin{equation}\label{e48}
\text{End}^n(E_*)\, \subset\, \text{End}(E)
\end{equation}
be the coherent analytic subsheaf defined by the conditions that $s(E_x) \, \subset\, {\mathcal L}^*(-S)_x$
and $s({\mathcal L}^*(-S)_x) \,=\, 0$ for all $x\,\in\, S$ lying in the domain of the local section $s$ of
$\text{End}(E)$ (see Lemma \ref{lem1}). Take any
$$
\phi\, \,\in\, \, H^0(X,\, \text{End}^n(E_*)\otimes K_X(S))\, .
$$
Let
\begin{equation}\label{e43}
\widehat{\phi}\,\,:\,\, {\mathcal L}\, \, \longrightarrow\, {\mathcal L}^*(-S)\otimes K_X(S)\,=\,{\mathcal L} 
\end{equation}
be the homomorphism given by the following composition of homomorphisms:
$$
{\mathcal L}\, \stackrel{\iota}{\longrightarrow}\, E \, \stackrel{\phi}{\longrightarrow}\, E\otimes K_X(S)\,
\xrightarrow{\,\,p\otimes{\rm Id}_{K_X(S)}\,}\, {\mathcal L}^*(-S)\otimes K_X(S) \,=\,{\mathcal L} \, ,
$$
where $\iota$ and $p$ are the homomorphisms in \eqref{e4}; recall that ${\mathcal L}^{\otimes 2}\,=\, K_X$.

\begin{proposition}\label{prop6}
For every $\phi\,\in\, H^0(X,\, {\rm End}^n(E_*)\otimes K_X(S))$ the homomorphism
$\widehat{\phi}$ constructed from it in \eqref{e43} vanishes identically.
\end{proposition}

\begin{proof}
Tensoring the diagram in \eqref{e6} with $K_X(S)$ we have the following commutative diagram
\begin{equation}\label{e44}
\begin{matrix}
0 & \longrightarrow & {\mathcal L}\otimes K_X &\longrightarrow & \widetilde{E}\otimes K_X
& \longrightarrow & {\mathcal L} & \longrightarrow & 0\\
&&\Big\downarrow && \,\,\, \Big\downarrow q &&\Big\downarrow\\
0 & \longrightarrow &{\mathcal L}\otimes K_X(S) & \longrightarrow & E\otimes K_X(S) 
&\longrightarrow & {\mathcal L} & \longrightarrow & 0.
\end{matrix}
\end{equation}
Take any $\phi\,\in\, H^0(X,\, \text{End}^n(E_*)\otimes K_X(S))$. Consider the composition of homomorphisms
$$
\widetilde{E}(-S)\, \stackrel{\psi}{\longrightarrow}\, E \, \stackrel{\phi}{\longrightarrow}\, 
E\otimes K_X(S)\, ,
$$
where $\psi$ is the homomorphism in \eqref{e6}, and denote this composition by $\widetilde{\phi}$. From
\eqref{e44}, \eqref{e48} and the construction of the decomposition in Lemma \ref{lem1} it follows that
the image of this homomorphism $\widetilde{\phi}\,:\, \widetilde{E}(-S)\,\longrightarrow\, E\otimes
K_X(S)$ is contained in the image of the homomorphism $q$ in \eqref{e44}; in other words, the subsheaf
$\phi\circ\psi (\widetilde{E}(-S))\,\subset\, E\otimes K_X(S)$ lies in the image of the homomorphism
$$
\psi\otimes {\rm Id}_{K_X(S)}\,\,:\,\, \widetilde{E}(-S)\otimes K_X(S)\,=\, \widetilde{E}\otimes K_X
\, \longrightarrow\, E\otimes K_X(S)\, .
$$
Consequently, $\phi$ produces a homomorphism
\begin{equation}\label{e45}
\phi'\,\, :\,\, \widetilde{E}(-S)\,\, \longrightarrow\,\, \widetilde{E}\otimes K_X\, .
\end{equation}
More precisely, $\phi'$ is determined uniquely by the condition
$$
\widetilde{\phi}\,=\, (\psi\otimes {\rm Id}_{K_X(S)})\circ\phi'.
$$
Let
\begin{equation}\label{e49}
\phi''\,\, :\,\, \mathcal{L}(-S)\,\,\longrightarrow\,\, \mathcal{L}
\end{equation}
denote the following composition of homomorphisms
$$
\mathcal{L}(-S)\,\stackrel{\iota'}{\longrightarrow}\, \widetilde{E}(-S) \, \stackrel{\phi'}{\longrightarrow}\,
\widetilde{E}\otimes K_X\, \xrightarrow{\, p_0\otimes{\rm Id}_{K_X}\, }\, {\mathcal L}^*\otimes K_X\,=\, {\mathcal L}\, ,
$$
where $\iota'$ and $p_0$ are the homomorphisms in \eqref{e6} and \eqref{e3} respectively. To prove the proposition
it suffices to show that $\phi''$ in \eqref{e49} vanishes identically.

Take any $x_i\, \in\, S$. Since $$q(\phi'(x_i)(\widetilde{E}(-S)_{x_i}))\,=\, \phi(\psi(x_i)(\widetilde{E}(-S)_{x_i}))
\,=\, \phi({\mathcal L}^*(-S)_{x_i})\,=\, 0\, ,$$
where $\psi$, $\phi'$ and $q$ are the homomorphisms in \eqref{e6}, \eqref{e45} and \eqref{e44} respectively, we conclude that
\begin{equation}\label{e46}
\phi'(x_i)(\widetilde{E}(-S)_{x_i})\, \subset\, ({\mathcal L}\otimes K_X)_{x_i}
\, \subset\, (\widetilde{E}\otimes K_X)_{x_i}\, ,
\end{equation}
where $\phi'$ is the homomorphism in \eqref{e45} and ${\mathcal L}\, \subset\, \widetilde{E}$ is the
subbundle in \eqref{e3}.

Furthermore, it can be shown that
\begin{equation}\label{e47}
\phi'(x_i)(\mathcal{L}(-S)_{x_i})\,=\, 0\, ;
\end{equation}
see \eqref{e6} for the subspace $\mathcal{L}(-S)_{x_i}\, \subset\, \widetilde{E}(-S)_{x_i}$. Indeed, this again follows
from \eqref{e6}, \eqref{e44}, \eqref{e48} and the construction of the decomposition in Lemma \ref{lem1}.

In view of \eqref{e46} and \eqref{e47}, the homomorphism $\phi''$ in \eqref{e49} vanishes at each $x_i$. Therefore,
$\phi''$ produces a homomorphism
\begin{equation}\label{e50}
\phi'''\,\, :\,\, \mathcal{L}(-S)\,\,\longrightarrow\,\, \mathcal{L}(-S)\, .
\end{equation}

Consider the image $\phi'(\mathcal{L}(-S))\, \subset\, \widetilde{E}\otimes K_X$, where $\phi'$ is the homomorphism in
\eqref{e45}. If the homomorphism $\phi'''$ in \eqref{e50} in nonzero, then this subsheaf $\phi'(\mathcal{L}(-S))$ produces
a holomorphic splitting of the top short exact sequence in \eqref{e6} tensored with $K_X$. Indeed, in that
case the homomorphism $p'\otimes {\rm Id}_{K_X}$ (see \eqref{e6} for $p'$) maps $\phi'(\mathcal{L}(-S))$ surjectively to
$\mathcal{L}^*(-S)\otimes K_X\,=\, \mathcal{L}(-S)$ and hence $\phi'(\mathcal{L}(-S))$ gives a holomorphic splitting
of the short exact sequence
$$
0 \, \longrightarrow \, {\mathcal L}(-S)\otimes K_X \,\longrightarrow\,\widetilde{E}(-S)\otimes K_X
\, \longrightarrow \, {\mathcal L}^*(-S)\otimes K_X \, \longrightarrow \, 0
$$
obtained from the top exact sequence in \eqref{e6} by tensoring it with $K_X$. A holomorphic splitting of the above
exact sequence produces a holomorphic splitting of the top short exact sequence in \eqref{e6}. But the exact sequence in \eqref{e3} does
not split holomorphically, which implies that the top short exact sequence in \eqref{e6} does not split
holomorphically. This implies that $\phi'''\,=\, 0$ (see \eqref{e50}), and
hence $\phi''\,=\, 0$ (see \eqref{e49}). As noted before, to prove the proposition
it is enough to show that $\phi''$ vanishes identically. This completes the proof.
\end{proof}

\begin{corollary}\label{cor5}
The endomorphism ${\mathcal S}(D_0,\, {\mathcal L})\, :\, {\mathcal L}\,\longrightarrow\, {\mathcal L}$ in Corollary
\ref{cor1}(2) does not depend on the parabolic connection $D_0$.
\end{corollary}

\begin{proof}
The space of parabolic connections on $E_*$ is an affine
space for the vector space $H^0(X,\, {\rm End}^n(E_*)\otimes K_X(S))$. Note that for any parabolic connection
$D$ on $E_*$ and any $\phi\,\in\, H^0(X,\, {\rm End}^n(E_*)\otimes K_X(S))$, we have
$$
{\mathcal S}(D+\phi,\, {\mathcal L})\,=\, {\mathcal S}(D,\, {\mathcal L})+ \widehat{\phi},
$$
where $\widehat{\phi}$ is constructed in \eqref{e43} from $\phi$. Therefore, from Proposition \ref{prop6} it
follows immediately that ${\mathcal S}(D+\phi,\, {\mathcal L})\,=\, {\mathcal S}(D,\, {\mathcal L})$.
\end{proof}

As before, let ${\mathcal L}_*$ denote the holomorphic line bundle $\mathcal L$ in \eqref{e4} equipped with
the parabolic structure on it induced by $E_*$ for the inclusion map $\iota$ in
\ref{e4}. We denote by $E_*/{\mathcal L}_*$ the quotient line
bundle $E/{\mathcal L}$ in \eqref{e4} equipped with the parabolic structure on it induced by $E_*$.
So from \eqref{e4} we have a short exact sequence of parabolic bundles
\begin{equation}\label{e37}
0\, \longrightarrow\, {\mathcal L}_*\, \longrightarrow\, E_* \, \longrightarrow\, 
E_*/{\mathcal L}_*\, \longrightarrow\, 0\, .
\end{equation}

For notational convenience, both $\text{Sym}^{0}(E_*)$ and $({\mathcal L}_*)^0$
will denote the trivial holomorphic line bundle ${\mathcal O}_X$
equipped with the trivial parabolic structure (no nonzero parabolic weights). Since
$\text{Sym}^{r-1}(E_*)$ is a quotient of $(E_*)^{\otimes (r-1)}$,
we have a natural homomorphism of parabolic bundles
$$
\tau_j\,\, :\, \,\text{Sym}^{j-1}(E_*)\otimes ({\mathcal L}_*)^{r-j}\, \longrightarrow\,
\text{Sym}^{r-1}(E_*)
$$
for every $1\,\leq\, j\, \leq\, r$ (see \eqref{e37}). This $\tau_j$ is an injective homomorphism, and
its image is a parabolic subbundle of $\text{Sym}^{r-1}(E_*)$. Let
$$
{\mathcal F}^j_* \,\,:=\, {\rm Image}(\tau_j) \, \subset\, \text{Sym}^{r-1}(E_*)
$$
be the parabolic subbundle; its rank is $j$. So we have a filtration of parabolic subbundles
\begin{equation}\label{e38}
0\,=\, {\mathcal F}^0_*\, \subset\, {\mathcal F}^1_*\, \subset\, {\mathcal F}^2_*\, \subset\, \cdots\,
\subset\, {\mathcal F}^{r-1}_*\, \subset\, {\mathcal F}^r_*\,=\, \text{Sym}^{r-1}(E_*).
\end{equation}
The holomorphic vector bundle underlying any ${\mathcal F}^i_*$ will be denoted by ${\mathcal F}^i_0$.

For any $1\, \leq\, j\, \leq\, r$, the quotient parabolic line bundle ${\mathcal F}^j_*/{\mathcal F}^{j-1}_*$
in \eqref{e38} actually has the following description:
\begin{equation}\label{e39}
{\mathcal F}^j_*/{\mathcal F}^{j-1}_*\,=\, ({\mathcal L}_*)^{r-j}\otimes (E_*/{\mathcal L}_*)^{j-1}\,.
\end{equation}
Indeed, this follows immediately from \eqref{e37}; by convention, $(E_*/{\mathcal L}_*)^0$ is the trivial line
bundle ${\mathcal O}_X$ with the trivial parabolic structure. It can be shown that
\begin{equation}\label{l1}
({\mathcal L}_*)^*\,=\, E_*/{\mathcal L}_* .
\end{equation}
Indeed, from \eqref{e10} it follows that ${\mathcal L}_*\otimes (E_*/{\mathcal L}_*) \, =\, \det E_*$ is the 
trivial line bundle ${\mathcal O}_X$ with the trivial parabolic structure, and hence
\eqref{l1} holds. Therefore, from \eqref{e39} it follows that
\begin{equation}\label{e40}
\text{par-deg}({\mathcal F}^j_*/{\mathcal F}^{j-1}_*)\,=\,
(2j-r-1)\cdot \text{par-deg}(E_*/{\mathcal L}_*)\,=\,
(2j-r-1)\cdot\left(1-g-n +\sum_{i=1}^n \frac{c_i+1}{2c_i+1}\right)\, ,
\end{equation}
where $g\,=\, \text{genus}(X)$. Now from \eqref{e38} and \eqref{e40} it is deduced that
\begin{equation}\label{e41}
\text{par-deg}({\mathcal F}^j_*)\,=\, \sum_{i=1}^j \text{par-deg}
({\mathcal F}^i_*/{\mathcal F}^{i-1}_*)\,=\,
j(r-j)\cdot\left(g-1 +\sum_{i=1}^n \frac{c_i}{2c_i+1}\right)\, .
\end{equation}

\begin{lemma}\label{lem6}
Let $D$ be any parabolic connection on the parabolic bundle ${\rm Sym}^{r-1}(E_*)$. Then the following two hold:
\begin{enumerate}
\item For any $1\,\leq\, j\, \leq\, r-1$, the parabolic subbundle ${\mathcal F}^{j}_*$ in \eqref{e38}
is not preserved by $D$.

\item $D({\mathcal F}^{j}_0)\, \subset\, {\mathcal F}^{j+1}_0\otimes K_X(S)$, where ${\mathcal F}^i_0$ is the
holomorphic vector bundle underlying ${\mathcal F}^i$, for all $1\,\leq\, j\, \leq\, r-1$.
\end{enumerate}
\end{lemma}

\begin{proof}
{}From \eqref{e41} it follows that $\text{par-deg}({\mathcal F}^j_*)\,\not=\, 0$ (in fact,
$\text{par-deg}({\mathcal F}^j_*)\,>\, 0$) for all $1\,\leq\, j\, \leq\,
r-1$. Consequently, $D$ does not preserve ${\mathcal F}^j_*$.

For any $1\,\leq\, j\, \leq\, r-2$, and any $2\, \leq\, k\,\leq\, r-j$, consider the parabolic line bundle 
$$
({\mathcal F}^j_*/{\mathcal F}^{j-1}_*)^*\otimes ({\mathcal F}^{j+k}_*/{\mathcal F}^{j+k-1}_*)\,=\,
(({\mathcal L}_*)^{r-j}\otimes (E_*/{\mathcal L}_*)^{j-1})^*\otimes
(({\mathcal L}_*)^{r-j-k} \otimes (E_*/{\mathcal L}_*)^{j+k-1})
$$
$$
=\, ({\mathcal L}_*)^{r-j-k- (r-j)} \otimes (E_*/{\mathcal L}_*)^{j+k-1- (j-1)}
\,=\, ({\mathcal L}_*)^{-k} \otimes (E_*/{\mathcal L}_*)^{k}\,=\, (E_*/{\mathcal L}_*)^{2k};
$$
see \eqref{e39} and \eqref{l1} for the above isomorphisms. The holomorphic line bundle underlying the parabolic
line bundle $({\mathcal F}^j_*/{\mathcal F}^{j-1}_*)^*\otimes ({\mathcal F}^{j+k}_*/{\mathcal F}^{j+k-1}_*)\,=\,
(E_*/{\mathcal L}_*)^{2k}$ will be denoted by $\xi_{r,k}$. We have
$$
\text{degree}(\xi_{r,k})\,=\, 2k\cdot \text{degree}(E/{\mathcal L}) + \sum_{i=1}^n\left[\frac{2k(c_i+1)}{2c_i+1}\right]
$$
$$
=\, 2k (1-g -n) + kn + \sum_{i=1}^n\left[\frac{k}{2c_i+1}\right] \,=\, k(2-2g-n)+ + \sum_{i=1}^n\left[\frac{k}{2c_i+1}\right]\, ,
$$
where $[t]\, \in\, \mathbb Z$ denotes the integral part of $t$, meaning $0\, \leq\, t-[t]\, <\, 1$.
This implies that
$$
\text{degree}(\xi_{r,k}) \, <\, 2-2g-n\,=\, - \text{degree}(K_X(S))
$$
(recall that $n\, \geq\, 3$ if $g\,=\, 0$), and hence $\text{degree}(\xi_{r,k}\otimes K_X(S))\, <\, 0$.
Consequently, we have
$$
H^0(X,\, \xi_{r,k}\otimes K_X(S))\,=\, 0\, .
$$
This implies that
\begin{equation}\label{t2}
H^0(X,\, ({\mathcal F}^j_*/{\mathcal F}^{j-1}_*)^*\otimes ({\mathcal F}^{j+k}_*/{\mathcal F}^{j+k-1}_*)\otimes K_X(S))
\,=\,0\, .
\end{equation}

{}From \eqref{t2} it is deduced that the following composition of homomorphisms
\begin{equation}\label{t1}
{\mathcal F}^{j}_0\, \stackrel{D}{\longrightarrow}\, {\mathcal F}^{r}_0\otimes K_X(S)\, \longrightarrow\,
({\mathcal F}^{r}_0/{\mathcal F}^{j+1}_0)\otimes K_X(S)
\end{equation}
vanishes identically, where ${\mathcal F}^{\ell}_0$ is the holomorphic vector bundle underlying the
parabolic bundle ${\mathcal F}^{\ell}_*$. To see this, observe that the parabolic vector bundle
$$\text{Hom}({\mathcal F}^{j}_*,\, ({\mathcal F}^{r}_*/{\mathcal F}^{j+1}_*)\otimes K_X(S))\,=\,
({\mathcal F}^{r}_*/{\mathcal F}^{j+1}_*)\otimes K_X(S)\otimes ({\mathcal F}^{j}_*)^*\,=\,
({\mathcal F}^{r}_*/{\mathcal F}^{j+1}_*)\otimes ({\mathcal F}^{j}_*)^*\otimes K_X(S)
$$
has a filtration of parabolic subbundles such that the successive quotients are
$$({\mathcal F}^j_*/{\mathcal F}^{j-1}_*)^*\otimes ({\mathcal F}^{j+k}_*/{\mathcal F}^{j+k-1}_*)
\otimes K_X(S)\, ,\ \ \, 2\,\leq\, k\, \leq\, r-j.$$ So \eqref{t2} implies that the
composition of homomorphisms in \eqref{t1} vanishes identically. Since the
composition of homomorphisms in \eqref{t1} vanishes identically we have
$$D({\mathcal F}^{j}_*)\, \subset\, {\mathcal F}^{j+1}_*$$ for all $1\,\leq\, j\, \leq\, r-1$.
\end{proof}

{}From \eqref{e39} it follows that for any $1\,\leq\, j\, \leq\, r-1$, the parabolic line bundle 
$$
({\mathcal F}^j_*/{\mathcal F}^{j-1}_*)^*\otimes ({\mathcal F}^{j+1}_*/{\mathcal F}^{j}_*)
\,=\, (E_*/{\mathcal L}_*) \otimes {\mathcal L}^*_*\,=\, (E_*/{\mathcal L}_*)^{\otimes 2}
$$
is $TX(-S)\,=\, K_X(S)^*$ equipped with the parabolic weight $\frac{1}{2c_i+1}$ at each $x_i\, \in\, S$
(see \eqref{l1} for the above isomorphism). Therefore, from
Lemma \ref{lem6}(2) we conclude that for any parabolic connection $D$ on the parabolic bundle ${\rm Sym}^{r-1}(E_*)$,
the second fundamental forms for the parabolic subbundles in \eqref{e38} are given by a collection of holomorphic homomorphisms
\begin{equation}\label{e42}
\psi(D,j)\, \,\in \,\, H^0(X,\, \text{Hom}({\mathcal F}^j_*/{\mathcal F}^{j-1}_*,\,
{\mathcal F}^{j+1}_*/{\mathcal F}^{j}_*)\otimes K_X(S))\,=\, H^0(X,\, {\mathcal O}_X)
\end{equation}
with $1\,\leq\, j\, \leq\, r-1$.

\begin{corollary}\label{cor4}
For each $1\,\leq\, j\, \leq\, r-1$, the section $\psi(D,j)$ in \eqref{e42} is a nonzero constant.
\end{corollary}

\begin{proof}
{}From Lemma \ref{lem6}(1) it follows immediately that $\psi(D,j)\, \not=\, 0$.
\end{proof}

\section{Differential operators on parabolic bundles} \label{se5}

In this section we will describe differential operators between parabolic vector bundles. As before, fix a compact Riemann surface
$X$ and a reduced effective divisor $S\,=\,\sum_{i=1}^n x_i$ on it; if $\text{genus}(X)\,=\, 0$, then assume that
$n\, \geq\, 3$. For each point $x_i\, \in\, S$ fix an integer
$N_i\, \geq\, 2$. We will consider parabolic bundles on $X$ with parabolic structure on $S$ such that all the parabolic
weights at each $x_i\,\in\, S$ are integral multiplies of $1/N_i$.

There is a ramified Galois covering
\begin{equation}\label{h1}
\varphi\, :\, Y\, \longrightarrow\, X
\end{equation}
satisfying the following two conditions:
\begin{itemize}
\item $\varphi$ is unramified over the complement $X\setminus S$, and

\item for every $x_i\, \in\, S$ and one (hence every) point $y\, \in\, \varphi^{-1}(x_i)$,
the order of the ramification of $\varphi$ at $y$ is $N_i$.
\end{itemize}
Such a ramified Galois covering $\varphi$ exists; see \cite[p. 26, Proposition 1.2.12]{Na}.
Let
\begin{equation}\label{h2}
\Gamma\,:=\, \text{Gal}(\varphi) \,:=\, \text{Aut}(Y/X) \, \subset\, \text{Aut}(Y)
\end{equation}
be the Galois group for $\varphi$. So the restriction
\begin{equation}\label{h3}
\varphi'\, :=\, \varphi\big\vert_{Y'}\,: \, Y' \,:=\, Y\setminus \varphi^{-1}(S) \,\,\longrightarrow\,\,
X'\,:=\, X\setminus S
\end{equation}
is an \'etale Galois covering with Galois group $\Gamma$.

As before, a holomorphic vector bundle $V$ on $Y$
is called an \textit{orbifold bundle} if $\Gamma$ acts on $V$ as holomorphic bundle
automorphisms over the action of $\Gamma$ on $Y$.

Consider the trivial vector bundle
\begin{equation}\label{e30}
{\mathbb C}[\Gamma]_Y\, :=\, Y\times {\mathbb C}[\Gamma] \, \longrightarrow\, Y\, ,
\end{equation}
where ${\mathbb C}[\Gamma]$ is the group algebra for $\Gamma$ with coefficients in $\mathbb C$. The
usual action of $\Gamma$ on ${\mathbb C}[\Gamma]$ and the Galois action of $\Gamma$ on $Y$ together
produce an action of $\Gamma$ on $Y\times {\mathbb C}[\Gamma]$. This action makes
$Y\times {\mathbb C}[\Gamma]\,=\, {\mathbb C}[\Gamma]_Y$ an orbifold bundle on $Y$. Let
\begin{equation}\label{e31}
{\mathcal E}_*\,\, \longrightarrow\,\, X
\end{equation}
be the corresponding parabolic vector bundle on $X$ with parabolic structure on $S$ \cite{Bi1},
\cite{Bo1}, \cite{Bo2}. The action of $\Gamma$ on the vector bundle ${\mathbb C}[\Gamma]_Y$
in \eqref{e30} produces an action of $\Gamma$ on its direct image $\varphi_*{\mathbb C}[\Gamma]_Y$
over the trivial action of $\Gamma$ on $X$. We have
\begin{equation}\label{fi2}
{\mathcal E}_0\,=\, (\varphi_*{\mathbb C}[\Gamma]_Y)^\Gamma\, \subset\, \varphi_*{\mathbb C}[\Gamma]_Y\, ,
\end{equation}
where $(\varphi_*{\mathbb C}[\Gamma]_Y)^\Gamma$ is the $\Gamma$--invariant part, and ${\mathcal E}_0$ is
the holomorphic vector bundle underlying the parabolic bundle ${\mathcal E}_*$ in \eqref{e31}.

It can be shown that
the holomorphic vector bundle ${\mathcal E}_0\,=\, (\varphi_*{\mathbb C}[\Gamma]_Y)^\Gamma$ is
identified with $\varphi_*{\mathcal O}_Y$. Indeed, there is a natural $\Gamma$--equivariant isomorphism
$$
\varphi_*{\mathbb C}[\Gamma]_Y
\,\, \stackrel{\sim}{\longrightarrow} \,\,
(\varphi_*{\mathcal O}_Y)\otimes_{\mathbb C} {\mathbb C}[\Gamma]\, ;
$$
it is in fact given by the projection formula. Therefore, the natural isomorphism
$$
\varphi_*{\mathcal O}_Y\,\, \stackrel{\sim}{\longrightarrow} \,\,
((\varphi_*{\mathcal O}_Y)\otimes_{\mathbb C} {\mathbb C}[\Gamma])^\Gamma
$$
(any complex $\Gamma$--module $M$ is naturally identified with $(M\otimes_{\mathbb C} {\mathbb C}[\Gamma])^\Gamma$)
produces an isomorphism
\begin{equation}\label{fi}
\varphi_*{\mathcal O}_Y\,\, \stackrel{\sim}{\longrightarrow} \,\,
(\varphi_*{\mathbb C}[\Gamma]_Y)^\Gamma .
\end{equation}

The direct image $\varphi_* {\mathcal O}_Y$ has a natural parabolic structure which we will now describe.

Take any $x_i\, \in\, S$. Fix an analytic open neighborhood $U\, \subset\, X$ of $x_i$ such that
$U\bigcap S \,=\, x_i$. Let ${\mathcal U}\, :=\, \varphi^{-1}(U)\, \subset\, Y$ be the inverse image.
The restriction of $\varphi$ to $\mathcal U$ will be denoted by $\widetilde{\varphi}$.
Let $\widetilde{D}_i\, :=\, \varphi^{-1}(x_i)_{\rm red}\, \subset\, Y$ be the reduced inverse
image. For all $k\, \in\, [1,\, N_i]$, define the vector bundle
$$
V_k\, :=\, \widetilde{\varphi}_*{\mathcal O}_{\mathcal U}(-(N_i-k))\widetilde{D}_i) \,\longrightarrow\, U\, .
$$
So we have a filtration of subsheaves of $V_{N_i}\,=\, (\varphi_* {\mathcal O}_Y)\big\vert_U$:
$$
0\,\subset\, V_1\, \subset\, V_2\, \subset\, \cdots\, \subset\, V_{N_i-1} \, \subset\, V_{N_i}
\,=\, (\varphi_* {\mathcal O}_Y)\big\vert_U\, .
$$
The restriction of this filtration of subsheaves to $x_i$ gives a filtration of subspaces
\begin{equation}\label{e32}
0\,\subset\, (V_1)'_{x_i}\, \subset\, (V_2)'_{x_i}\, \subset\, \cdots\, \subset\, (V_{N_i-1})'_{x_i}
\, \subset\, (V_{N_i})_{x_i} \,=\, (\varphi_* {\mathcal O}_Y)_{x_i}
\end{equation}
of the fiber $(\varphi_* {\mathcal O}_Y)_{x_i}$. We note that $(V_k)'_{x_i}$ in \eqref{e32} is the image, in
the fiber $(\varphi_* {\mathcal O}_Y)_{x_i}$, of the fiber $(V_k)_{x_i}$ over $x_i$ of the vector bundle $V_k$.

The parabolic structure on $\varphi_* {\mathcal O}_Y$ is defined as follows. The parabolic 
divisor is $S$. The quasiparabolic filtration over any $x_i\,\in\, S$ is the filtration of 
$(\varphi_* {\mathcal O}_Y)_{x_i}$ constructed in \eqref{e32}. The parabolic weight of the 
subspace $(V_k)_{x_i}$ in \eqref{e32} is $\frac{N_i-k}{N_i}$. The resulting parabolic vector 
bundle is identified with ${\mathcal E}_*$ in \eqref{e31}; recall from \eqref{fi2} and \eqref{fi} 
that ${\mathcal E}_0$ is identified with $\varphi_*{\mathcal O}_Y$.

The trivial connection on the trivial vector bundle ${\mathbb C}[\Gamma]_Y\, :=\, Y\times {\mathbb C}[\Gamma]$
in \eqref{e30} is preserved by the action of the Galois group $\Gamma$ on ${\mathbb C}[\Gamma]_Y$. Therefore,
this trivial connection produces a parabolic connection on the corresponding parabolic vector bundle ${\mathcal E}_*$ in
\eqref{e31}. This parabolic connection on ${\mathcal E}_*$ will be denoted by $\nabla^{\mathcal E}$.

Using the isomorphism between ${\mathcal E}_0$ and $\varphi_*{\mathcal O}_Y$ (see \eqref{fi2} and \eqref{fi}), the
logarithmic connection on ${\mathcal E}_0$ defining the above parabolic connection $\nabla^{\mathcal E}$ on ${\mathcal E}_*$
produces a logarithmic connection on $\varphi_*{\mathcal O}_Y$. This logarithmic connection on $\varphi_*{\mathcal O}_Y$
given by $\nabla^{\mathcal E}$ is easy to describe. To describe it, take the de Rham differential $d\,:\, {\mathcal O}_Y\,
\longrightarrow\, K_Y$ on $Y$. Let
\begin{equation}\label{k1}
\varphi_*d \, :\, \varphi_*{\mathcal O}_Y\, \longrightarrow\, \varphi_*K_Y
\end{equation}
be its direct image. On the other hand, using the projection formula, the natural homomorphism 
$$
K_Y\, \hookrightarrow\, K_Y\otimes {\mathcal O}_Y(\varphi^{-1}(S)_{\rm red})\,=\, \varphi^*(K_X\otimes{\mathcal O}_X(S))\, .
$$
produces a homomorphism
$$
\varphi_* K_Y\,\, \longrightarrow\,\, \varphi_*(\varphi^*(K_X\otimes{\mathcal O}_X(S)))\,
=\, (\varphi_*{\mathcal O}_Y)\otimes K_X\otimes{\mathcal O}_X(S)\, .
$$
Combining this with $\varphi_*d$ in \eqref{k1} we obtain homomorphisms
$$
\varphi_*{\mathcal O}_Y\, \longrightarrow\, \varphi_*K_Y
\,\, \longrightarrow\,\, (\varphi_*{\mathcal O}_Y)\otimes K_X\otimes{\mathcal O}_X(S)\, .
$$
This composition of homomorphisms $\varphi_*{\mathcal O}_Y\,\longrightarrow\, (\varphi_*{\mathcal O}_Y)\otimes K_X\otimes{\mathcal O}_X(S)$
defines a logarithmic connection on $\varphi_*{\mathcal O}_Y$. This logarithmic connection coincides
with the one that defines the above constructed parabolic connection $\nabla^{\mathcal E}$ on ${\mathcal E}_*$.

The parabolic connection $\nabla^{\mathcal E}$ on ${\mathcal E}_*$ defines a nonsingular holomorphic connection ${\nabla}'$ on
$$
{\mathcal E}'_0\, :=\, {\mathcal E}_0\big\vert_{X'}\,=\,\varphi_{1*}{\mathcal O}_{Y'}
$$
over $X'$ (see \eqref{h3}). For any holomorphic vector bundle $V'$ on $X'$, note that
\begin{equation}\label{e33}
J^k(V'\otimes {\mathcal E}'_0) \,=\, J^k(V')\otimes{\mathcal E}'_0
\end{equation}
for all $k\, \geq\, 0$. To see this isomorphism, for any $x\, \in\, X'$ and $u\, \in\, ({\mathcal E}'_0)_x$,
let $\widetilde{u}$ denote the unique flat section of ${\mathcal E}'_0$ for the connection $\nabla'$, defined on any simply connected
open neighborhood of $x$, such that $\widetilde{u}(x)\,=\, u$. Now the homomorphism
$$
J^k(V')\otimes{\mathcal E}'_0\, \longrightarrow\, J^k(V'\otimes {\mathcal E}'_0)
$$
that sends any $v\otimes u$ to the image of $v\otimes\widetilde{u}$, where $v\, \in\, J^k(V')_x$ and $u\, \in\, ({\mathcal E}'_0)_x$ with
$x\, \in\, X'$, is evidently an isomorphism.

Take holomorphic vector bundles $V'$ and $W'$ on a nonempty Zariski open subset $U\, \subset\, X'$. Recall that a holomorphic
differential operator of order $k$ from $V'$ to $W'$ is a holomorphic homomorphism $J^k(V')\, \longrightarrow\, W'$. Let
$$
D'\,:\, J^k(V')\, \longrightarrow\, W'
$$
be a holomorphic differential operator of order $k$ from $V'$ to $W'$ on $U$.

We will show that $D'$ extends to a
holomorphic differential operator
\begin{equation}\label{e34}
\widetilde{D'}\, :\, J^k(V'\otimes {\mathcal E}'_0)\, \longrightarrow\, W'\otimes {\mathcal E}'_0
\end{equation}
from $V'\otimes {\mathcal E}'_0$ to $W'\otimes {\mathcal E}'_0$ over $U$. To construct
$\widetilde{D'}$, using the isomorphism in \eqref{e33} we have
$$
J^k(V'\otimes {\mathcal E}'_0) \,=\, J^k(V')\otimes{\mathcal E}'_0\,\,\,\xrightarrow{\,\, D'\otimes {\rm Id}_{{\mathcal E}'_0}
\,\,} \,\,\, W'\otimes {\mathcal E}'_0\, .
$$
This homomorphism is the one in \eqref{e34}.

Let $V_*$ and $W_*$ be parabolic vector bundles on $X$. Denote the restrictions $V_0\big\vert_{X'}$ and $W_0\big\vert_{X'}$
by $V'$ and $W'$ respectively. The holomorphic vector bundle underlying the parabolic tensor product
$V_*\otimes {\mathcal E}_*$ (respectively, $W_*\otimes {\mathcal E}_*$) will be denoted by
$(V_*\otimes {\mathcal E}_*)_0$ (respectively, $(W_*\otimes {\mathcal E}_*)_0$), where ${\mathcal E}_*$ is the
parabolic bundle in \eqref{e31}.

\begin{definition}\label{def2}
A {\it holomorphic differential operator} of order $k$ from $V_*$ to $W_*$ over an open subset
$\widetilde{U}\,\subset\, X$ is a holomorphic homomorphism
$$
D'\,:\, J^k(V')\, \longrightarrow\, W'
$$
over $U\,:=\, \widetilde{U}\bigcap X'$ such that the homomorphism 
$$
\widetilde{D'}\, :\, J^k(V'\otimes {\mathcal E}'_0)\, \longrightarrow\, W'\otimes {\mathcal E}'_0
$$
in \eqref{e34} extends to a holomorphic homomorphism
$J^k((V_*\otimes {\mathcal E}_*)_0)\, \longrightarrow\, (W_*\otimes {\mathcal E}_*)_0$ over entire $\widetilde{U}$.
\end{definition}

It is straightforward to check that the above definition does not depend on the choice of the map $\varphi$.

We denote by ${\rm Diff}^k_X(V_*,\, W_*)$ the sheaf of holomorphic differential operators of order 
$k$ from $V_*$ to $W_*$. Define
$$
{\rm DO}^k_P(V_*,\, W_*)\,\,:=\,\, H^0(X,\, {\rm Diff}^k_X(V_*,\, W_*))
$$
to be the space of all holomorphic 
differential operators of order $k$ from $V_*$ to $W_*$ over $X$.

Let $\mathbb V$ and $\mathbb W$ denote the orbifold vector bundles on $Y$ corresponding to the parabolic vector
bundles $V_*$ and $W_*$ respectively. Consider the space
$$
{\rm DO}^k(\mathbb{V},\, \mathbb{W})\,:=\, H^0(Y,\, \text{Hom}(J^k(\mathbb{V}),\, \mathbb{W}))
$$
of holomorphic differential operators of order $k$ from $\mathbb V$ to
$\mathbb W$ over $Y$. Then the actions of $\Gamma$ on $\mathbb V$ and $\mathbb W$ together produce an
action of $\Gamma$ on ${\rm DO}^k(\mathbb{V},\, \mathbb{W})$. Let
$$
H^0(Y,\, \text{Hom}(J^k(\mathbb{V}),\, \mathbb{W}))^\Gamma\,=\, {\rm DO}^k(\mathbb{V},\, \mathbb{W})^\Gamma
\, \subset\, {\rm DO}^k(\mathbb{V},\, \mathbb{W})
$$
be the space of all $\Gamma$--invariant differential operators of order $k$ from $\mathbb V$ to $\mathbb W$.

\begin{proposition}\label{prop5}
There is a natural isomorphism
$$
{\rm DO}^k(\mathbb{V},\, \mathbb{W})^\Gamma\,\, \stackrel{\sim}{\longrightarrow}\,\, 
{\rm DO}^k_P(V_*,\, W_*)\, .
$$
\end{proposition}

\begin{proof}
We will first prove that
\begin{equation}\label{e35}
\varphi_*\mathbb{V}\,=\, (V_*\otimes {\mathcal E}_*)_0\, ,
\end{equation}
where ${\mathcal E}_*$ is the parabolic bundle in \eqref{e31} and
$(V_*\otimes {\mathcal E}_*)_0$ is the vector bundle underlying the parabolic
vector bundle $V_*\otimes {\mathcal E}_*$. To prove \eqref{e35}, first note that
\begin{equation}\label{e36}
\varphi_*\mathbb{V}\,=\, \left(\varphi_*({\mathbb V}\otimes {\mathbb C}[\Gamma]_Y)\right)^\Gamma\, ,
\end{equation}
where ${\mathbb C}[\Gamma]_Y$ is the orbifold bundle in \eqref{e30}. Since $\mathcal{E}_*$ and
$V_*$ correspond to the orbifold bundles ${\mathbb C}[\Gamma]_Y$ and $\mathbb V$
respectively, the parabolic bundle corresponding to the orbifold bundle
${\mathbb V}\otimes {\mathbb C}[\Gamma]_Y$ is $V_*\otimes {\mathcal E}_*$. In particular, we have
$$
\left(\varphi_*({\mathbb V}\otimes {\mathbb C}[\Gamma]_Y)\right)^\Gamma \,=\, 
(V_*\otimes {\mathcal E}_*)_0\, .
$$
This and \eqref{e36} together give the isomorphism in \eqref{e35}.

Let $D\, :\, \mathbb{V}\, \longrightarrow\, \mathbb{W}$ be a holomorphic differential operator of order $k$
on $Y$. Taking its direct image for the map $\varphi$, we have
$$
\varphi_* D\,\,:\,\, \varphi_*\mathbb{V} \,\, \longrightarrow\,\,\varphi_*\mathbb{W}\, .
$$
Now if $D\, \in\, {\rm DO}^k(\mathbb{V},\, \mathbb{W})^\Gamma$, then clearly
$$
\varphi_* D((\varphi_*\mathbb{V})^\Gamma)\,\, \subset\,\, (\varphi_*\mathbb{W})^\Gamma\, .
$$
Let
$$
D_\varphi\,\,:=\,\,(\varphi_* D)\big\vert_{(\varphi_*\mathbb{V})^\Gamma}
\,\,:\,\, (\varphi_*\mathbb{V})^\Gamma \,\, \longrightarrow\,\,(\varphi_*\mathbb{W})^\Gamma
$$
be the restriction of $\varphi_* D$ to $(\varphi_*\mathbb{V})^\Gamma\, \subset\,
\varphi_*\mathbb{V}$.

Using \eqref{e35} it is now straightforward to check that $D_\varphi$ defines a 
holomorphic differential operator of order $k$ from the parabolic bundle $V_*$ to $W_*$. The
corresponding homomorphism $J^k((V_*\otimes {\mathcal E}_*)_0)\, \longrightarrow\,
(W_*\otimes {\mathcal E}_*)_0$ in Definition \ref{def1} is given by $\varphi_* D$ using the
isomorphism in \eqref{e35}.

The isomorphism in the proposition sends any $D\, \in\, {\rm DO}^k(\mathbb{V},\, \mathbb{W})^\Gamma$ to
$D_\varphi\, \in\, {\rm DO}^k_P(V_*,\, W_*)$ constructed above from $D$.

For the inverse map, given any $\mathbf{D}\, \in\, {\rm DO}^k_P(V_*,\, W_*)$, consider the
homomorphism 
$$J^k((V_*\otimes {\mathcal E}_*)_0)\, \longrightarrow\,
(W_*\otimes {\mathcal E}_*)_0$$ in Definition \ref{def1} given by the differential operator
$\mathbf{D}$. Using the isomorphism in \eqref{e35} it produces a holomorphic differential operator from
$\mathbb V$ to $\mathbb W$. This differential operator is evidently fixed by
the action of $\Gamma$ on ${\rm DO}^k(\mathbb{V},\, \mathbb{W})$.
\end{proof}

\subsection{Another description of differential operators on parabolic bundles}

We will give an alternative description of the holomorphic differential operators between two
parabolic vector bundles. Let ${\rm Diff}^k_Z(A,\, B)$ denote the sheaf of holomorphic differential operators 
of order $k$ from a holomorphic vector bundle $A$ on a complex manifold $Z$ to another holomorphic vector
bundle $B$ on $Z$. The sheaf ${\rm Diff}^k_Z({\mathcal O}_Z,\, {\mathcal O}_Z)\,=\, J^k({\mathcal O}_Z)^*$ has
both left and right ${\mathcal O}_Z$--module structures, and
\begin{equation}\label{e51}
{\rm Diff}^k_Z(A,\, B)\,\,=\,\, B\otimes_{{\mathcal O}_Z}
{\rm Diff}^k_Z({\mathcal O}_Z,\, {\mathcal O}_Z)
\otimes_{{\mathcal O}_Z} A^*\, .
\end{equation}
We have a short exact sequence of holomorphic vector bundles
\begin{equation}\label{e52}
0\, \longrightarrow\, {\rm Diff}^k_Z({\mathcal O}_Z,\, {\mathcal O}_Z)
\, \stackrel{\alpha}{\longrightarrow}\, {\rm Diff}^{k+1}_Z({\mathcal O}_Z,\, {\mathcal O}_Z)
\, \stackrel{\eta}{\longrightarrow}\, \text{Sym}^{k+1}(TZ) \, \longrightarrow\, 0\, ,
\end{equation}
where $\eta$ is the symbol map. The homomorphism
$$
{\rm Id}_B\otimes\alpha\otimes {\rm Id}_{A^*}\, :\,
B\otimes_{{\mathcal O}_Z}{\rm Diff}^k_Z({\mathcal O}_Z,\, {\mathcal O}_Z)\otimes_{{\mathcal O}_Z} A^*
\, \longrightarrow\, B\otimes_{{\mathcal O}_Z}{\rm Diff}^{k+1}_Z({\mathcal O}_Z,\, {\mathcal O}_Z)
\otimes_{{\mathcal O}_Z} A^*\, ,
$$
where $\alpha$ is the homomorphism in \eqref{e52}, coincides with the natural inclusion
map $${\rm Diff}^k_Z(A,\, B)\, \hookrightarrow\,{\rm Diff}^{k+1}_Z(A,\, B).$$

The holomorphic differential operators between two parabolic vector bundles will be described 
along the above line.

Consider the pair $(Y,\, \varphi)$ in \eqref{h1}. The action of $\Gamma\,=\, 
\text{Gal}(\varphi)$ on $Y$ produces an action of $\Gamma$ on ${\mathcal O}_Y$. This
action of $\Gamma$ on ${\mathcal O}_Y$ induces an action of $\Gamma$ on
$J^k({\mathcal O}_Y)$, which in turn induces an action of $\Gamma$ on the dual vector
bundle $J^k({\mathcal O}_Y)^*\,=\, {\rm Diff}^k_Y({\mathcal O}_Y,\, {\mathcal O}_Y)$.
As mentioned before, ${\rm Diff}^k_Y({\mathcal O}_Y,\, {\mathcal O}_Y)$ is equipped with
left and right ${\mathcal O}_Y$--module structures. These module structures are $\Gamma$--equivariant.
Let ${\mathcal J}^k_*$ denote the parabolic vector bundle on $X$ associated to the
orbifold vector bundle $J^k({\mathcal O}_Y)^*\,=\, {\rm Diff}^k_Y({\mathcal O}_Y,\, {\mathcal O}_Y)$
on $Y$.

Note that the rank of ${\mathcal J}^k_*$ is $k+1$. The parabolic line bundle ${\mathcal J}^0_*$
is the trivial line bundle ${\mathcal O}_X$ equipped with the trivial parabolic structure.
The underlying holomorphic vector bundle for the parabolic bundle ${\mathcal J}^1_*$ is
${\mathcal O}_X\oplus TX(-S)$. The quasiparabolic filtration of ${\mathcal J}^1_*$ over
any point $x_i\, \in\, S$ is
$$
TX(-S)_{x_i}\, \subset\, ({\mathcal O}_X)_{x_i}\oplus TX(-S)_{x_i}\,=\, ({\mathcal J}^1_0)_{x_i}\, .
$$
The parabolic weight of $TX(-S)_{x_i}$ is $\frac{1}{N_i}$ and the parabolic weight of
$({\mathcal J}^1_0)_{x_i}$ is $0$. Let
\begin{equation}\label{e53}
TX(-S)_*\, \, \longrightarrow\, X
\end{equation}
denote the parabolic line bundle defined by $TX(-S)$ equipped with the
parabolic weight $\frac{1}{N_i}$ at each $x_i\, \in\, S$. So
$$
{\mathcal J}^1_*\,=\, TX(-S)_*\oplus {\mathcal O}_X,
$$
where ${\mathcal O}_X$ has the trivial parabolic structure.

Using the homomorphism $\alpha$ in \eqref{e52} for $Y$ and $j\,=\, k$ we see that
${\mathcal J}^j_*$ is a parabolic subbundle of ${\mathcal J}^{j+1}_*$ for
all $j\, \geq\, 0$. Consequently, we have filtration of parabolic subbundles
\begin{equation}\label{e55}
{\mathcal J}^0_*\, \subset\, {\mathcal J}^1_*\, \subset\, \cdots \, \subset\, 
{\mathcal J}^{k-1}_* \, \subset\, {\mathcal J}^k_*
\end{equation}
for all $k\, \geq\, 0$ such that each successive quotient is a parabolic line bundle.

We will describe
the quotient parabolic line bundle ${\mathcal J}^{j}_*/{\mathcal J}^{j-1}_*$ in \eqref{e55}
for all $1\, \leq\, j\,\leq\, k$.

The holomorphic line bundle underlying the parabolic bundle ${\mathcal J}^{j}_*/{\mathcal J}^{j-1}_*$
is
$$
(TX)^{\otimes j}(-jS)\otimes{\mathcal O}_X
\left(\sum_{i=1}^n\left[\frac{j}{N_i}\right]x_i\right)\, ,
$$
where $\left[\frac{j}{N_i}\right]\, \in\, \mathbb Z$ is the integral part of $\frac{j}{N_i}$,
and its parabolic weight at any $x_i\, \in\, S$ is $\frac{j}{N_i} -\left[\frac{j}{N_i}\right]$.
Indeed, from \eqref{e52} we know that the parabolic line bundle ${\mathcal J}^{j}_*/
{\mathcal J}^{j-1}_*$ corresponds to the orbifold line bundle $(TY)^{\otimes j}$ on $Y$. On the other
hand, the parabolic line bundle $TX(-S)_*$ defined in \eqref{e53} corresponds to the orbifold line
bundle $TY$. Therefore, we have
\begin{equation}\label{e54}
{\mathcal J}^{j}_*/{\mathcal J}^{j-1}_*\,=\, TX(-S)^{\otimes j}_*\, .
\end{equation}
The above description of ${\mathcal J}^{j}_*/{\mathcal J}^{j-1}_*$ follows immediately from \eqref{e54}.

The $\Gamma$--equivariant left and right ${\mathcal O}_Y$--module structures on
${\rm Diff}^k_Y({\mathcal O}_Y,\, {\mathcal O}_Y)$ produces left
and right ${\mathcal O}_X$--module structures on ${\mathcal J}^k_*$. 

Then, for any two parabolic bundles $V_*$ and $W_*$ over $X$, it follows from Proposition \ref{prop5} and \eqref{e51} that 
${\rm Diff}^k_X(V_*,\, W_*)$ coincides with the holomorphic vector bundle underlying the
parabolic tensor product
$$
W_*\otimes_{{\mathcal O}_X} {\mathcal J}^k_*\otimes_{{\mathcal O}_X}
V^*_*\, ;
$$
in other words, we have
$$
{\rm Diff}^k_X(V_*,\, W_*)\,=\, (W_*\otimes_{{\mathcal O}_X} {\mathcal J}^k_*\otimes_{{\mathcal O}_X}
V^*_*)_0\, .
$$

\subsection{The symbol map}

Consider the quotient map
$$
\gamma\,\, :\,\, {\mathcal J}^k_*\, \longrightarrow\, {\mathcal J}^k_*/{\mathcal J}^{k-1}_*\,=\,
TX(-S)^{\otimes k}_*
$$
(see \eqref{e55}, \eqref{e54}). It produces a map
\begin{equation}\label{e56}
\sigma\, :=\, ({\rm Id}_{W_*}\otimes\gamma\otimes{\rm Id}_{V^*_*})_0\,:\,
{\rm Diff}^k_X(V_*,\, W_*)\,=\, (W_*\otimes_{{\mathcal O}_X} {\mathcal J}^k_*
\otimes_{{\mathcal O}_X}\otimes V^*_*)_0
\end{equation}
$$
\longrightarrow\, (W_*\otimes TX(-S)^{\otimes k}_*\otimes V^*_*)_0\,=\,
(TX(-S)^{\otimes k}_*\otimes {\rm Hom}(V_*,\, W_*)_*)_0\, .
$$
The above homomorphism $\sigma$ is the \textit{symbol} map of differential operators between
parabolic bundles.

Take any $\widehat{D}\, \in\, {\rm DO}^k_P(V_*,\, W_*)$. Denote by $\mathbb{V}$ (respectively,
$\mathbb{W}$) the orbifold bundle on $Y$ corresponding to $V_*$ (respectively, $W_*$), and let
$$
D\, \in\, {\rm DO}^k(\mathbb{V},\, \mathbb{W})^\Gamma
$$
be the invariant differential operator given by $\widehat{D}$ using Proposition \ref{prop5}. Let
$$
\sigma(\widehat{D})\, \in\, H^0(X,\, (TX(-S)^{\otimes k}_*\otimes {\rm Hom}(V_*,\, W_*)_*)_0)
$$
be the symbol of $\widehat{D}$ (see \eqref{e56}). Let
$$
\sigma(D)\, \in\, H^0(Y,\, \text{Hom}(\mathbb{V},\, \mathbb{W})\otimes (TY)^{\otimes k})
$$
be the symbol of $D$. We have
$$
\sigma(D)\, \in\, H^0(Y,\, \text{Hom}(\mathbb{V},\, \mathbb{W})\otimes (TY)^{\otimes k})^\Gamma
$$
because $D$ is fixed by the action of $\Gamma$ on ${\rm DO}^k(\mathbb{V},\, \mathbb{W})$. The
proof of the following lemma is straightforward.

\begin{lemma}\label{lem7}
The parabolic vector bundle $TX(-S)^{\otimes k}_*\otimes {\rm Hom}(V_*,\, W_*)_*$ on $X$ corresponds
to the orbifold vector bundle $\text{Hom}(\mathbb{V},\, \mathbb{W})\otimes (TY)^{\otimes k}$ on $Y$.
The natural isomorphism
$$
H^0(X,\, (TX(-S)^{\otimes k}_*\otimes {\rm Hom}(V_*,\, W_*)_*)_0)\, \stackrel{\sim}{\longrightarrow}
\, H^0(Y,\, {\rm Hom}(\mathbb{V},\, \mathbb{W})\otimes (TY)^{\otimes k})^\Gamma
$$
takes the symbol $\sigma(\widehat{D})$ to the symbol $\sigma(D)$.
\end{lemma}

\section{Parabolic opers and differential operators}

Recall the short exact sequence in \eqref{e37} and the isomorphism in \eqref{l1}.
For notational convenience, $({\mathcal L}_*)^*\,=\, E_*/{\mathcal L}_*$ will be
denoted by ${\mathcal L}^{-1}_*$. For any $j\, \leq\, 1$, the parabolic line bundle
$({\mathcal L}_*)^{\otimes j}$ (respectively, $({\mathcal L}^*_*)^{\otimes j}$)
will be denoted by ${\mathcal L}^j_*$ (respectively, ${\mathcal L}^{-j}_*$).
Also, ${\mathcal L}^{0}_*$ will denote the trivial line bundle ${\mathcal O}_X$ with the
trivial parabolic structure.

We note that
\begin{equation}\label{l2}
{\mathcal L}^{-2}_*\,=\, TX(-S)_*\, ,
\end{equation}
where $TX(-S)_*$ is the parabolic line bundle in \eqref{e53}. From \eqref{e54} and \eqref{l2}
it follows that
\begin{equation}\label{l4}
{\mathcal J}^{j}_*/{\mathcal J}^{j-1}_*\,=\, {\mathcal L}^{-2j}_*
\end{equation}
for all $j\, \geq\, 1$.

For any integer $r\, \geq\, 2$, consider the space of parabolic differential operators of order $r$
$$
{\rm DO}^r_P({\mathcal L}^{1-r}_* ,\, {\mathcal L}^{r+1}_*)\,:=\,
H^0(X,\, {\rm Diff}^r_X({\mathcal L}^{1-r}_*,\, {\mathcal L}^{r+1}_*))
$$
from ${\mathcal L}^{1-r}_*$ to ${\mathcal L}^{r+1}_*$. Let
\begin{equation}\label{l3}
\sigma\,:\, {\rm DO}^r_P({\mathcal L}^{1-r}_* ,\, {\mathcal L}^{r+1}_*)\,\longrightarrow\,
({\mathcal L}^{r+1}_* \otimes (TX(-S)_*)^{\otimes r}\otimes {\mathcal L}^{r-1}_*)_0
\end{equation}
$$
=\,({\mathcal L}^{r+1}_* \otimes {\mathcal L}^{-2r}_*\otimes {\mathcal L}^{r-1}_*)_0
\,=\,({\mathcal L}^0_*)_0\,=\, {\mathcal O}_X
$$
be the symbol map constructed in \eqref{e56} (see \eqref{l4} for the isomorphism used in \eqref{l3}).

Let
\begin{equation}\label{l5}
\widetilde{\rm DO}^r_P({\mathcal L}^{1-r}_* ,\, {\mathcal L}^{r+1}_*)\,\subset\,
{\rm DO}^r_P({\mathcal L}^{1-r}_* ,\, {\mathcal L}^{r+1}_*)
\end{equation}
be the affine subspace consisting of parabolic differential operators whose symbol is
the constant function $1$.

The following Lemma constructs the sub-principal symbol of the operator:

\begin{lemma}\label{lem8}
There is a natural map
$$
\Psi\, :\, \widetilde{\rm DO}^r_P({\mathcal L}^{1-r}_* ,\,
{\mathcal L}^{r+1}_*)\,\longrightarrow\, H^0(X,\, K_X)\, .
$$
\end{lemma}

\begin{proof}
As in \eqref{e17}, let $\mathbf{L}$ denote the orbifold line bundle on $Y$ corresponding
to ${\mathcal L}$. So the parabolic bundle ${\mathcal L}^{1-r}_*$ (respectively, ${\mathcal L}^{r+1}_*$)
corresponds to the orbifold line bundle $\mathbf{L}^{1-r}$ (respectively, $\mathbf{L}^{r+1}$).
Take any $$D\, \in\, \widetilde{\rm DO}^r_P({\mathcal L}^{1-r}_* ,\, {\mathcal L}^{r+1}_*).$$
Now Proposition \ref{prop5} says that $D$ corresponds to a $\Gamma$--invariant holomorphic differential
operator of order $r$ from $\mathbf{L}^{1-r}$ to $\mathbf{L}^{r+1}$. Let
\begin{equation}\label{a1}
{\mathcal D}\,\in\,{\rm DO}^r(\mathbf{L}^{1-r},\, \mathbf{L}^{r+1})^\Gamma
\end{equation}
be the $\Gamma$--invariant differential operator corresponding to $D$. As the orbifold bundle
$\mathbf{L}^2$ is isomorphic to $TY$ (see Lemma \ref{lem2}), the symbol of ${\mathcal D}$ is a
section of ${\mathcal O}_Y$. Since the symbol of
$D$ is the constant function $1$, from Lemma \ref{lem7} it follows that the symbol
of $\mathcal D$ is the constant function $1$ on $Y$.

We will now show that a differential operator ${\mathbf D}\, \in\, {\rm DO}^r(\mathbf{L}^{1-r},\, \mathbf{L}^{r+1})$
of symbol $1$ produces a section
\begin{equation}\label{ds}
\theta_{\mathbf D}\, \in\, H^0(Y,\, K_Y)\, .
\end{equation}

Consider the short exact sequence of jet bundles
\begin{equation}\label{js}
0 \, \longrightarrow\, \mathbf{L}^{1-r}\otimes K^{\otimes r}_Y\,=\, \mathbf{L}^{r+1}\,
\stackrel{\mu}{\longrightarrow}\, J^r(\mathbf{L}^{1-r})\, \stackrel{\nu}{\longrightarrow}\,
J^{r-1}(\mathbf{L}^{1-r})\, \longrightarrow\, 0
\end{equation}
(see Lemma \ref{lem2} for the above isomorphism) together with the homomorphism
$$
{\mathbf D}'\,:\, J^r(\mathbf{L}^{1-r}) \, \longrightarrow\, \mathbf{L}^{r+1}
$$
defining the given differential operator $\mathbf D$. Since the symbol of $\mathbf D$ is $1$, we have
$$
{\mathbf D}'\circ\mu\, =\, {\rm Id}_{\mathbf{L}^{r+1}}\, ,
$$
where $\mu$ is the homomorphism in \eqref{js}. Therefore, ${\mathbf D}'$ produces a holomorphic splitting
of the short exact sequence in \eqref{js}. Let
\begin{equation}\label{js3}
\tau\,:\, J^{r-1}(\mathbf{L}^{1-r})\, \longrightarrow\, J^{r}(\mathbf{L}^{1-r})
\end{equation}
be the holomorphic homomorphism given by this splitting of the short exact sequence in \eqref{js}, so $\tau$
is uniquely determined by the following two conditions:
\begin{itemize}
\item $\nu\circ\tau\,=\, {\rm Id}_{J^{r-1}(\mathbf{L}^{1-r})}$, where $\nu$ is the projection in \eqref{js}, and

\item ${\rm image}(\tau)\,=\, \text{kernel}({\mathbf D}')\, \subset\, J^r(\mathbf{L}^{1-r})$.
\end{itemize}
Next consider the following natural commutative diagram of homomorphisms of jet bundles:
\begin{equation}\label{js2}
\begin{matrix}
&& 0 && 0 && 0\\
&& \Big\downarrow && \Big\downarrow && \Big\downarrow\\
0 & \longrightarrow & \mathbf{L}^{1-r}\otimes K^{\otimes r}_Y\,=\, \mathbf{L}^{r+1} &
\stackrel{\mu}{\longrightarrow} & J^r(\mathbf{L}^{1-r}) & \stackrel{\nu}{\longrightarrow} &
J^{r-1}(\mathbf{L}^{1-r}) & \longrightarrow & 0\\
&& \Big\downarrow && \,\,\, \Big\downarrow\varpi && \Big\Vert\\
0 & \longrightarrow & J^{r-1}(\mathbf{L}^{1-r})\otimes K_Y &
\longrightarrow & J^1(J^{r-1}(\mathbf{L}^{1-r})) & \stackrel{\alpha}{\longrightarrow} &
J^{r-1}(\mathbf{L}^{1-r}) & \longrightarrow & 0\\
&& \Big\downarrow && \,\,\,\Big\downarrow\zeta &&\\
0 & \longrightarrow & J^{r-2}(\mathbf{L}^{1-r})\otimes K_Y &
\stackrel{=}{\longrightarrow} & J^{r-2}(\mathbf{L}^{1-r})\otimes K_Y &&\\
&& \Big\downarrow && \Big\downarrow &&\\
&& 0 && 0 &&
\end{matrix}
\end{equation}
where the horizontal sequences are the natural jet sequences, and the vertical sequence in the left is
the jet sequence tensored with $K_Y$; the homomorphism $\varpi$ is the natural homomorphism of jet bundles.
The homomorphism $\zeta$ in \eqref{js2} is constructed as follows:
We have the natural homomorphism $$h_1\,:\, J^1(J^{r-1}(\mathbf{L}^{1-r}))\, \longrightarrow\,
J^{r-1}(\mathbf{L}^{1-r}).$$ On the other hand, we have the composition of homomorphisms
$$
J^1(J^{r-1}(\mathbf{L}^{1-r}))\, \longrightarrow\, J^1(J^{r-2}(\mathbf{L}^{1-r}))\,
\longrightarrow\, J^{r-1}(\mathbf{L}^{1-r}),
$$
which will be denoted by $h_2$. Now, we have $\zeta\,=\, h_1-h_2$; note that $J^{r-2}(\mathbf{L}^{1-r})\otimes K_Y$
is a subbundle of $J^{r-1}(\mathbf{L}^{1-r})$.

Next consider the homomorphism
$$
\varpi\circ\tau\,:\, J^{r-1}(\mathbf{L}^{1-r})\, \longrightarrow\, J^1(J^{r-1}(\mathbf{L}^{1-r}))\, ,
$$
where $\tau$ and $\varpi$ are the homomorphisms in \eqref{js3} and \eqref{js2} respectively. We have
\begin{equation}\label{el}
\alpha\circ (\varpi\circ\tau)\,=\, {\rm Id}_{J^{r-1}(\mathbf{L}^{1-r})}\, ,
\end{equation}
where $\alpha$ is the projection in \eqref{js2}, because \eqref{js2} is a commutative diagram.

{}From \eqref{el} it follows immediately that $\varpi\circ\tau$ gives a holomorphic splitting
of the bottom exact sequence in \eqref{js2}. But a holomorphic splitting
of the bottom exact sequence in \eqref{js2} is a holomorphic connection on $J^{r-1}(\mathbf{L}^{1-r})$.

Let $\nabla$ denote the holomorphic connection on $J^{r-1}(\mathbf{L}^{1-r})$ given by $\varpi\circ\tau$.
The holomorphic connection on $\bigwedge^r J^{r-1}(\mathbf{L}^{1-r})\,=\, {\mathcal O}_Y$ (see Lemma
\ref{lem2}) induced by $\nabla$ will be denoted by $\nabla^0$. So the connection $\nabla^0$ is of the form
$$
\nabla^0\,=\, d+\theta_{\mathbf D}\, ,
$$
where $\theta_{\mathbf D}\, \in\, H^0(Y,\, K_Y)$ and $d$ is the de Rham differential on ${\mathcal O}_Y$. This
$\theta_{\mathbf D}$ is the holomorphic $1$-form in \eqref{ds}.

By the construction of it, the form $\theta_{\mathbf D}$ vanishes identically if and only if
the above connection $\nabla$ on 
$J^{r-1}(\mathbf{L}^{1-r})$ induces the trivial connection on $\bigwedge^r J^{r-1}(\mathbf{L}^{1-r})\,=\,
{\mathcal O}_Y$. Therefor $\theta_{\mathbf D}$ should be seen as a sub-principal symbol.

Consider $\theta_{\mathcal D}\, \in\, H^0(Y,\, K_Y)$ (as in \eqref{ds}) for the differential operator
$\mathcal D$ in \eqref{a1}. Since $\mathcal D$ is $\Gamma$--invariant, we know that $\theta_{\mathcal D}$
is also $\Gamma$--invariant. On the other hand,
$$
H^0(Y,\, K_Y)^\Gamma\,=\, H^0(X,\, K_X)\, .
$$
The element of $H^0(X,\, K_X)$ corresponding to $\theta_{\mathcal D}$ will be denoted by
$\theta'_{\mathcal D}$.

Now we have a map $$\Psi\,:\, \widetilde{\rm DO}^r_P({\mathcal L}^{1-r}_* ,\, {\mathcal L}^{r+1}_*)
\,\longrightarrow\, H^0(X,\, K_X)$$ that sends any $D$ to $\theta'_{\mathcal D}$ constructed above from $D$.
\end{proof}

The following main Theorem deals with the space of all parabolic ${\rm SL}(r,{\mathbb C})$--opers on $X$ (see Definition \ref{def1}) with 
given singular set $S\,:=\, \{x_1,\, \cdots,\, x_n\}\, \subset\, X$ and fixed integers $c_i\,=\, N_i$ (see \eqref{e2}).

\begin{theorem}\label{thm1}
The space of all parabolic ${\rm SL}(r,{\mathbb C})$--opers on $X$ is identified with the inverse image
$$\Psi^{-1}(0)\, \subset\, \widetilde{\rm DO}^r_P({\mathcal L}^{1-r}_* ,\, {\mathcal L}^{r+1}_*),$$
where $\Psi$ is the map in Lemma \ref{lem8}.
\end{theorem}

\begin{proof}
This theorem will be proved using Proposition \ref{prop4}, Proposition \ref{prop5}, Lemma \ref{lem7} and Lemma \ref{lem8}.

As before, fix a ramified Galois covering
$$
\varphi\, :\, Y\, \longrightarrow\, X
$$
satisfying the following two conditions:
\begin{itemize}
\item $\varphi$ is unramified over the complement $X\setminus S$, and

\item for every $x_i\, \in\, S$ and one (hence every) point $y\, \in\, \varphi^{-1}(x_i)$,
the order of ramification of $\varphi$ at $y$ is $2N_i+1$.
\end{itemize}
As before, $\Gamma$ denotes $\text{Aut}(Y/X)$.
Parabolic ${\rm SL}(r,{\mathbb C})$--opers on $X$ are in a natural bijective correspondence with the 
equivariant ${\rm SL}(r,{\mathbb C})$--opers on $Y$ (see Proposition \ref{prop4}). Equivariant ${\rm SL}(r,{\mathbb C})$--opers on $Y$ are in a
natural bijective correspondence with the
subspace of ${\mathcal D}\,\in\,{\rm DO}^r(\mathbf{L}^{1-r},\, \mathbf{L}^{r+1})^\Gamma$
(see \eqref{a1}) defined by all invariant differential operators $D$ satisfying the following two conditions:
\begin{itemize}
\item the symbol of $D$ is the constant function $1$, and

\item the element in $H^0(Y,\,K_Y)$ corresponding to $D$ (see \eqref{ds}) vanishes (this is equivalent
to the vanishing of the sub-principal symbol of $D$; see \cite[p.~13]{BD1}).
\end{itemize}
(See Proposition \ref{prop5} and Lemma \ref{lem7}.)

This subspace of ${\rm DO}^r(\mathbf{L}^{1-r},\, \mathbf{L}^{r+1})^\Gamma$ is 
in a natural bijective correspondence with $$\Psi^{-1}(0)\, \subset\, \widetilde{\rm DO}^r_P({\mathcal L}^{1-r}_* ,\,
{\mathcal L}^{r+1}_*),$$ where $\Psi$ is the map in Lemma \ref{lem8}.
\end{proof}

\section*{Acknowledgements}

We are very grateful to the referee for helpful comments.
This work has been supported by the French government through the UCAJEDI Investments in the Future 
project managed by the National Research Agency (ANR) with the reference number ANR2152IDEX201. The 
first author is partially supported by a J. C. Bose Fellowship, and school of mathematics, TIFR, is 
supported by 12-R$\&$D-TFR-5.01-0500.



\begin{thebibliography}{ZZZZZZ}

\bibitem[AB]{AB} D. G. L. Allegretti and Tom Bridgeland, The monodromy of meromorphic 
projective structures, \textit{Trans. Amer. Math. Soc.} {\bf 373} (2020), 6321--6367.

\bibitem[ABF]{ABF} M. Alim, F. Beck and L. Fredrickson, Parabolic Higgs bundles, $tt^*$ 
connections and opers, arXiv:1911.06652v1.

\bibitem[At]{At} M. F. Atiyah, Complex analytic connections in fibre
bundles, \textit{Trans. Amer. Math. Soc.} \textbf{85} (1957), 181--207.

\bibitem[BBP]{BBP} V. Balaji, I. Biswas and Y. Pandey, Connections on parahoric
torsors over curves, {\it Publ. Res. Inst. Math. Sci.} {\bf 53} (2017), 551--585. 

\bibitem[BD1]{BD1} A. Beilinson and V. G. Drinfeld, Opers, arXiv:0501398.

\bibitem[BD2]{BD2} A. Beilinson and V. G. Drinfeld, Quantization of Hitchin's integrable system
and Hecke eigensheaves, (1991).

\bibitem[BF]{BF} D. Ben-Zvi and E. Frenkel, Spectral curves, opers and integrable systems,
{\it Publ. Math. Inst. Hautes \'Etudes Sci.} {\bf 94} (2001), 87--159.

\bibitem[Bi]{Bi1} I. Biswas, Parabolic bundles as orbifold bundles, {\it Duke Math. J.}
{\bf 88} (1997), 305--325. 

\bibitem[BDP]{BDP} I. Biswas, S. Dumitrescu and C. Pauly, Parabolic
${\rm SL}_r$-opers, {\it Illinois J. Math.} {\bf 64} (2020), 493--517.

\bibitem[BDHP]{BDHP} I. Biswas, S. Dumitrescu, S. Heller and C. Pauly, Infinitesimal 
deformations of parabolic connections and parabolic opers, arXiv:2202.09125

\bibitem[BL]{BL} I. Biswas and M. Logares, Connection on parabolic vector bundles over
curves, {\it Inter. Jour. Math.} {\bf 22} (2011), 593--602.

\bibitem[BSY]{BSY} I. Biswas, L. P. Schaposnik and M. Yang, Generalized B-opers.
{\it Symmetry Integrability Geom. Methods Appl.} {\bf 16} (2020), Article 041.

\bibitem[Bi1]{Bi} I. Biswas, Parabolic bundles as orbifold bundles, {\it Duke Math. 
Jour.} {\bf 88} (1997), 305--325.

\bibitem[Bo1]{Bo1} N. Borne, Fibr\'es paraboliques et champ des racines, {\em Int. Math. 
Res. Not. IMRN}, {\bf 16}, Art. ID rnm049, 38, (2007).

\bibitem[Bo2]{Bo2} N. Borne, Sur les repr\'esentations du groupe fondamental d'une 
vari\'et\'e priv\'ee d'un diviseur \`a croisements normaux simples, {\em Indiana Univ. 
Math. Jour.} {\bf 58} (2009), 137--180.

\bibitem[CS]{CS} B. Collier and A. Sanders, (G,P)-opers and global Slodowy slices,
{\it Adv. Math.} {\bf 377} (2021), Paper No. 107490, 43 pp.

\bibitem[De]{De} P. Deligne, {\it \'Equations diff\'erentielles \`a points singuliers 
r\'eguliers}, Lecture Notes in Mathematics, Vol. 163, Springer-Verlag, Berlin-New York, 1970.

\bibitem[DS1]{DS1} V.~G.~Drinfeld and V.~V.~Sokolov, Lie algebras and equations of Korteweg-de
Vries type, Current problems in mathematics, Vol. 24, 81--180,
Itogi Nauki i Tekhniki, Akad. Nauk SSSR, Vsesoyuz. Inst. Nauchn. i Tekhn. Inform., Moscow, 1984.

\bibitem[DS1]{DS2} V.~G.~Drinfeld and V.~V.~Sokolov, Equations of Korteweg-de Vries type, and
simple Lie algebras, {\it Dokl. Akad. Nauk SSSR} {\bf 258} (1981), 11--16.

\bibitem[DFK+]{DFKMMN} O. Dumitrescu, L. Fredrickson, G. Kydonakis, R. Mazzeo, M. Mulase
and A. Neitzke, From the Hitchin section to opers through nonabelian Hodge,
{\it J. Differential Geom.} {\bf 117} (2021), 223--253.

\bibitem[Fr1]{Fr} E. Frenkel, Gaudin model and opers, {\it Infinite
dimensional algebras and quantum integrable systems}, 1--58, Progr. Math.,
237, Birkh\"auser, Basel, 2005.

\bibitem[Fr2]{Fr2} E. Frenkel, Lectures on the Langlands program and conformal field theory,
{\it Frontiers in number theory, physics, and geometry. II}, 387--533, Springer, Berlin, 2007.

\bibitem[FG1]{FG} E. Frenkel and D. Gaitsgory, Local geometric Langlands correspondence
and affine Kac-Moody algebras, {\it Algebraic geometry and number theory}, 69--260,
Progr. Math., 253, Birkh\"auser Boston, Boston, MA, 2006.

\bibitem[FG2]{FG2} E. Frenkel and D. Gaitsgory, Weyl modules and opers without monodromy,
{\it Arithmetic and geometry around quantization}, 101--121, Progr. Math., 279,
Birkh\"auser Boston, Boston, MA, 2010.

\bibitem[FT]{FT} E. Frenkel and C. Teleman, Geometric Langlands correspondence near opers,
{\it J. Ramanujan Math. Soc.} {\bf 28} (2013), 123--147.

\bibitem[GH]{GH} P. Griffiths and J. Harris, {\it Principles of algebraic geometry},
Pure and Applied Mathematics, Wiley-Interscience, New York, 1978.

\bibitem[Gu]{Gu} R. C. Gunning, {\it On uniformization of complex manifolds: the 
role of connections}, Princeton Univ. Press, 1978.

\bibitem[In]{In} M. Inaba, Moduli of parabolic connections on a curve and
Riemann-Hilbert correspondence, {\it J. Algebraic Geom.} {\bf 22} (2013), 407--480.

\bibitem[IIS1]{IIS1} M. Inaba, K. Iwasaki and M.-H. Saito,
Dynamics of the sixth Painlevé equation,
{\it Th\'{e}ories asymptotiques et \'{e}quations de Painlev\'{e}}, 103--167,
Sémin. Congr., 14, Soc. Math. France, Paris, 2006. 

\bibitem[IIS2]{IIS2} M. Inaba, K. Iwasaki and M.-H. Saito,
Moduli of stable parabolic connections, Riemann-Hilbert correspondence and geometry
of Painlev\'e equation of type VI. I.,
{\it Publ. Res. Inst. Math. Sci.} {\bf 42} (2006), 987--1089.

\bibitem[MY]{MY} M. Maruyama and K. Yokogawa, Moduli of parabolic stable sheaves, Math.
Ann. 293 (1992), 77--99.

\bibitem[MS]{MS} V. B. Mehta and C. S. Seshadri, Moduli of vector bundles on curves with
parabolic structures, \textit{Math. Ann.} \textbf{248} (1980), 205--239.

\bibitem[Na]{Na} M.~Namba, {\it Branched coverings and algebraic
functions}, Pitman Research Notes in Mathematics Series, 161,
Longman Scientific $\&$ Technical, Harlow; John Wiley $\&$ Sons,
Inc., New York, 1987.

\bibitem[Sa]{Sa} A. Sanders, The pre-symplectic geometry of opers and the holonomy map,
preprint, arXiv:1804.04716.

\bibitem[Wa]{W} Y. Wakabayashi, A theory of dormant opers on pointed stable
curves --- a proof of Joshi's conjecture, {\em arXiv:math.AG/1411.1208}.

\bibitem[Yo]{Yo} K. Yokogawa, Infinitesimal Deformation of Parabolic Higgs sheaves,
{\it Inter. Jour. Math.} {\bf 6} (1995), 125--148.

\end{thebibliography}
\end{document}